\documentclass[twoside,11pt]{article} 
\usepackage[english]{babel}
\usepackage{graphicx}
\usepackage{float}
\usepackage{amsmath}
\usepackage{bbold,color}
\usepackage{amssymb}
\usepackage{amsthm}
\usepackage{MnSymbol}
\usepackage[round]{natbib}
\usepackage{wasysym}
\usepackage{amsbsy}
\usepackage{parskip}
\usepackage[letterpaper, margin=1in]{geometry}

\newcommand{\qnvr}{q_n^*(\theta)}

\newcommand{\nX}{\mathbf X_n}
\newcommand{\nZ}{ z_{1:n}}

\newcommand{\cQ}{\mathcal{Q}}
\newcommand{\scKL}{\textsc{KL}}

\newcommand{\argmin}{\text{argmin}}
\newcommand{\argmax}{\text{argmax}}
\newcommand{\klvb}{{KL-VB}}

\newcommand{\red}[1]{\textcolor{red}{#1}}
\newcommand{\blue}[1]{\textcolor{blue}{#1}}
\renewcommand{\red}[1]{\unskip} 
\renewcommand{\blue}[1]{\textcolor{black}{#1}}

\newcommand{\good}{good sequence}
\newcommand{\Good}{Good sequence}

\def\bb1{\mathbb{I}}

\def\sN{\mathcal N}

\def\sQ{\mathcal Q}

\def\sZ{\mathcal Z}

\def\bbE{\mathbb E}

\def\bbR{\mathbb R}

\def\bk1{{\rm I\kern-.60em 1}}
\def\bkE{{\rm I\kern-.15em E}}
\def\bkL{{\rm I\kern-.15em L}}
\def\bkM{{\rm I\kern-.20em M}}
\def\bkN{{\rm I\kern-.15em N}}
\def\bkO{{\rm \kern.24em
            \vrule width.05em height1.4ex depth-.05ex
            \kern-.26em O}}\def\bkP{{\rm I\kern-.17em P}}
\def\bkQ{{\rm \kern.24em
            \vrule width.05em height1.4ex depth-.05ex
            \kern-.26em Q}}
\def\bkR{{\rm I\kern-.15em R}}
\def\bkZ{{\rm \kern-.15em Z}}

\def\a{\alpha}
\def\b{\beta}

\def\e{\epsilon}

 \newtheorem{definition}{Definition}[section]
 \newtheorem{lemma}{Lemma}[section]
 \newtheorem{theorem}{Theorem}[section]
 
 \newtheorem{proposition}{Proposition}[section]
\newtheorem{assumption}{Assumption}
 \newtheorem{example}{Example}[section]
\usepackage{lastpage}

\title{Asymptotic Consistency of $\alpha$-R\'enyi-Approximate Posteriors}
\author{Prateek Jaiswal{$^\star$}, Vinayak A. Rao{$^\dag$} and Harsha Honnappa{$^\star$}}
\date{$^\star$\{jaiswalp,honnappa\}@purdue.edu School of Industrial Engineering, Purdue University\\$^\dag$\{varao@purdue.edu\} Department of Statistics, Purdue University}

\begin{document}

\maketitle

\begin{abstract}%
We study the asymptotic consistency properties of $\alpha$-R\'enyi approximate posteriors, a class of variational Bayesian methods that approximate an intractable Bayesian posterior with a member of a tractable family of distributions, the member chosen to minimize the $\alpha$-R\'enyi divergence from the true posterior.\ Unique to our work is that we consider settings with $\alpha > 1$, resulting in approximations that {upperbound} the log-likelihood, and consequently have wider spread than traditional variational approaches that minimize the Kullback-Liebler (KL) divergence from the posterior. 
Our primary result identifies sufficient conditions under which consistency holds, centering around the existence of a `good' sequence of distributions in the approximating family that possesses, among other properties, the right rate of convergence to a limit distribution. We further characterize the good sequence by demonstrating that a sequence of distributions that converges too quickly cannot be a good sequence. 
\blue{We also extend our analysis to the setting where $\alpha$ equals one, corresponding to the minimizer of the reverse KL divergence, and to models with local latent variables}.  
We also illustrate the existence of~\good~with a number of examples. 
Our results complement a growing body of work focused on the frequentist properties of variational Bayesian methods.  
\end{abstract}
 

\section{Introduction}
Bayesian statistics forms a powerful and flexible framework that allows practitioners to bring prior knowledge to statistical problems, and to coherently manage uncertainty resulting from finite and noisy datasets. A Bayesian represents the unknown state of the world with a possibly vector-valued parameter $\theta$, over which they place a prior probability $\pi(\theta)$, representing {\em a priori} beliefs they might have. $\theta$ can include global parameters shared across the entire dataset, as well as local variables specific to each observation. A likelihood $p(\nX|\theta)$ then specifies a probability distribution over the observed dataset $\nX$.  Given observations $\nX$, prior beliefs $\pi(\theta)$ are updated to a posterior distribution $\pi(\theta|\nX)$ calculated through Bayes' rule. 

While conceptually straightforward, computing $\pi(\theta|\nX)$ is intractable for many interesting and practical models, and the field of Bayesian computation is focused on developing scalable and accurate computational techniques to approximate the posterior distribution. Traditionally, much of this has involved Monte Carlo and Markov chain Monte Carlo techniques to construct sampling approximations to the posterior distribution. In recent years, developments from machine learning have sought to leverage tools from optimization to construct tractable posterior approximations. An early and still popular instance of this methodology is {\em variational Bayes} (VB)~\citep{BlKuMc2017}.

At a high level, the idea behind VB is to approximate the intractable posterior $\pi(\theta|\nX)$ with an element $q(\theta)$ of some simpler class of distributions $\cQ$. Examples of $\cQ$ include the family of Gaussian distributions, delta functions, or the family of factorized `mean-field' distributions that discard correlations between components of $\theta$. The variational solution $q$ is the element of $\cQ$ that is closest to $\pi(\theta|\nX)$, where closeness is measured in terms of the Kullback-Leibler (KL) divergence. Thus, $q$ is the solution to: 
\begin{eqnarray}
  q(\theta) = \text{argmin}_{\tilde{q} \in \mathcal{Q}} 
  \text{KL}(\tilde{q}(\theta)\|\pi(\theta|\nX)). \label{eq:vb_opt}
\end{eqnarray}
We term this as the \klvb\ method. From the non-negativity of the KL divergence, we can view this as \blue{maximizing} a lower-bound to the logarithm of the \blue{model \textit{evidence}} \red{marginal probability of the observations }, $\log p(\nX) = \log \left(\int p(\nX,\theta) \mathrm{d} \theta \right)$. This lower-bound, called the variational lower-bound or evidence lower bound (ELBO) is defined as 
\begin{align}
  \text{ELBO}(\tilde{q}(\theta)) = \log p(\nX) - \text{KL}(\tilde{q}(\theta)\|p(\theta|\nX)).\label{eq:vb_opt2}
\end{align}
Optimizing the two equations above with respect to $q$ does not involve either calculating expectations with respect to the intractable posterior $\pi(\theta|\nX)$, or evaluating the posterior normalization constant. As a consequence, a number of standard optimization algorithms can be used to select the best approximation $q(\theta)$ to the posterior distribution, examples including expectation-maximization~\citep{neal1998view} and gradient-based~\citep{kingma2013auto} methods. 
This has allowed the application of Bayesian methods to increasingly large datasets and high-dimensional settings. Despite their widespread popularity in the machine learning, and more recently, the statistics communities,  it is only recently that variational \blue{Bayesian} methods have been studied theoretically~\citep{alquier2017concentration,abdellatif2018,WaBl2017,yang2017alpha, ZhGa2017}. 

  \subsection{R\'enyi Divergence Minimization}
  Despite its popularity, variational Bayes has a number of well-documented limitations. An important one is its tendency to produce approximations that underestimate the spread of the posterior distribution~\citep{TuSh2011a,LiTu2016}: in essence, the variational Bayes solution tends to match closely with the dominant mode of the posterior. This arises from the choice of the divergence measure $\text{KL}(q(\theta)\|\pi(\theta|\nX)) = \mathbb{E}_q[\log(q(\theta)/\pi(\theta|\nX))]$, which does not penalize solutions where $q(\theta)$ is small while $\pi(\theta|\nX)$ is large. While many statistical applications only focus on the mode of the distribution, definite calculations of the variance and higher moments are critical in predictive and decision-making problems. 

  A natural solution is to consider different divergence measures than those used in variational Bayes. Expectation propagation (EP)~\citep{Mi2001a} was developed to minimize $\mathbb{E}_p[\log(p/q)]$ instead, though this requires an expectation with respect to the intractable posterior. Consequently, EP can only minimize an approximation of this objective. \red{We will call  $\mathbb{E}_p[\log(p/q)]$ the `idealized' EP objective, see~\citet{wainwright2008graphical} for the actual EP loss function.} 
  
More recently, R\'enyi's $\alpha$-divergence~\citep{TErvan2012} has been used as a family of parametrized divergence measures for variational inference \citep{LiTu2016, dieng2017variational}. The $\alpha$-R\'enyi divergence is defined as 
\vspace{-1.5em}
 \begin{align*}
 D_{\alpha} \left( \pi(\theta| \mathbf X_n)\| {q}(\theta) \right):= \frac{1}{\alpha-1} 
  \log \int_{\Theta} {q}(\theta)   \left(\frac{\pi(\theta| \mathbf X_n)}{{q}(\theta) } \right)^{\alpha} d\theta. 
  \end{align*}
  The parameter $\alpha$ spans a number of divergence measures and, in particular, we note that as $\alpha \to 1$ we recover the \red{idealized} EP objective $\text{KL}(\pi(\theta|\nX) \| q(\theta))$, \blue{we will call its minimizer $1-$R\'enyi  approximate posterior}. 
  Settings of $\alpha > 1$ are particularly interesting since, in contrast to VB which lower-bounds the log-likelihood of the data~\eqref{eq:vb_opt2}, one obtains tractable upper bounds. Precisely, using Jensen's inequality,
\begin{align*}
  p(\nX)^\alpha  &= 
\left(  \int p(\theta,\nX)~ \frac{q(\theta)}{q(\theta)}~ d\theta\right)^\alpha 
\le \mathbb{E}_q \left[ \left(\frac{ p(\theta,\nX)}{q(\theta)} \right)^\alpha\right].
\end{align*}
{Applying the logarithm function on either side,}
\begin{align}
  \alpha \log p(\nX)  & \le \log
  \mathbb{E}_q \left[ \left(\frac{ \pi(\theta,\nX)}{q(\theta)} \right)^\alpha\right]\\ 
                    &= \alpha \log p(\nX) + \log \mathbb{E}_q \left[ \left(\frac{ \pi(\theta|\nX)}{q(\theta)} \right)^\alpha\right]  := \mathcal{F}_2(q).
                    \label{eq:renyi}
\end{align}
  Observe that the second term in the expression for $\mathcal{F}_2(q)$ 
  is just $(\alpha-1) D_\alpha(p(\theta|\nX)\|q(\theta))$. Like with the ELBO lower bound, evaluating this upper bound only involves expectations with respect to $q(\theta)$, and only requires evaluating $p(\theta,\nX)$, the unnormalized posterior distribution. Optimizing this upper bound over some class of distributions $\mathcal Q$, we obtain the {\it $\alpha$-R\'enyi} approximation. As noted before, standard variational Bayes, which optimizes a lower-bound, tends to produce approximating distributions that underestimate the posterior variance, resulting in predictions that are overconfident and ignore high-risk regions in the support of the posterior.  We illustrate this in Figure~\ref{fig:1} below that reproduces a result from~\citet{LiTu2016}. The true posterior distribution is an anisotropic Gaussian distribution and the variational family consists of isotropic (or mean-field) Gaussian distributions. Standard \klvb, represented by the curve $\alpha = 0$, clearly fits the mode of the posterior, but completely underestimates the dominant eigen-direction. On the other hand, for large values of $\alpha$ (shown as $\alpha \to +\infty$), the $\alpha$-R\'enyi approximate posterior matches the mode and does a better job of capturing the spread of the posterior. The figure also presents results for the $\alpha = 1$ and the $\alpha  \to -\infty$ cases. As an aside, we observe that our parametrization of the R\'enyi divergence is different from~\citet{LiTu2016}, where the upper-bounds considered in\blue{~\citet{LiTu2016}} \red{this paper} emerge as $\alpha \to -\infty$.
  \begin{figure}[H]
  	\centering
    \includegraphics[scale=0.35]{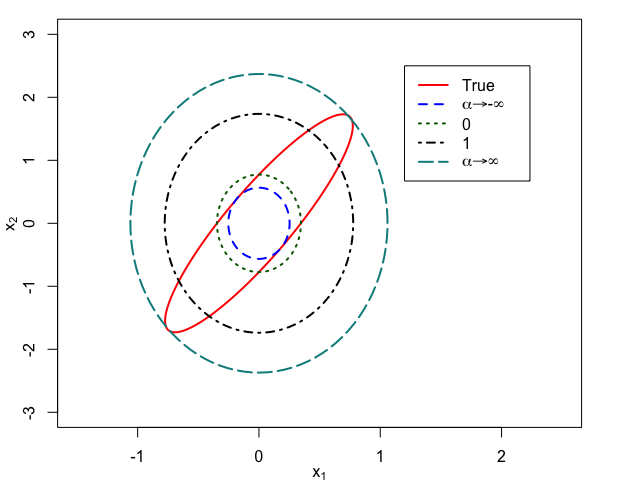}
        \caption{ Isotropic variational $\alpha$-R\'enyi approximations to an anisotropic Gaussian, for different values of $\alpha$ (see also \citet{LiTu2016}).}
  	\label{fig:1}
  \end{figure}
  We note, furthermore, that in tasks such as model selection, the marginal likelihood of the data is of fundamental interest~\citep{grosse2015sandwiching}, and the $\alpha$-R\'enyi upper bound provides an approximation that complements the VB lower bound.
Recent developments in stochastic optimization have allowed the $\alpha$-R\'enyi objective to be optimized fairly easily; see~\citet{LiTu2016} and \citet{ dieng2017variational}.

\subsection{Large Sample Properties}
Despite often state-of-the-art empirical results, variational methods still present a number of unanswered theoretical questions. This is particularly true for $\alpha$-R\'enyi divergence minimization which has empirically demonstrated very promising results for a number of applications~\citep{LiTu2016,dieng2017variational}. In recent work,~\citet{ZhGa2017} have shown conditions under which $\alpha$-R\'enyi variational methods are consistent when $\alpha$ is less than one. Their results followed from a proof for the regular Kullback-Leibler variational algorithm, and thus only apply to situations when a {\em lower-bound} is optimized. As we mentioned before, the setting with $\alpha$ greater than $1$ is qualitatively different from both Kullback-Leibler and R\'enyi divergence with $\alpha < 1$. This setting, which is also of considerable practical interest, is the focus of our paper and we address the question of asymptotic consistency of the approximate posterior distribution obtained by minimizing the R\'enyi divergence. 

Asymptotic consistency~\citep{vdV00} is a basic frequentist requirement of any statistical method, guaranteeing that the `true' parameter is recovered as the number of observations tends to infinity. Table~\ref{tab:1} summarizes the current known results on consistency of VI and EP, and highlights the gap that this paper is intended to fill. \blue{We note that in this work, we are not analyzing the actual EP algorithm~\citep{wainwright2008graphical}, and are instead looking at the global minimizer of the ideal EP objective.} 
 	
 	\begin{table}[hbt]
 		\centering
  		\begin{tabular}{|c|c|}
  		 	\hline
   			 Methods& Papers \\
   			\hline
	   			\klvb & \citet{WaBl2017},\citet{ZhGa2017} \\
                $\alpha$-R\'enyi ($\alpha < 1$) & \citet{ZhGa2017}\\
                $\alpha$-R\'enyi ($\alpha > 1$) & This paper\\   	
                \blue{$1$-R\'enyi~($\alpha \to 1$,  global EP )}  & This paper\\
   			 \hline
  		\end{tabular}
  		\caption{Known results on the asymptotic consistency of variational methods.}
	  	\label{tab:1}
	\end{table}
 	
As we will see, filling these gaps will require new developments. This follows from two complicating factors: 1) R\'enyi divergence with $\alpha > 1$ {\em upper-bounds} the log-likelihood, and 2) this requires new analytical approaches involving expectations with respect to the intractable $\pi(\theta|\nX)$. We thus emphasize that the results in our paper are not a consequence of recent analysis in~\citet{WaBl2017} and~\citet{ZhGa2017} for the~\klvb, and our proofs differ substantially from theirs.
	
We establish our main result in Theorem~\ref{thm:consistency} under mild regularity conditions. First, in Assumption~\ref{assume:prior} we assume that the prior distribution places positive mass in the neighborhood of the true parameter $\theta_0$, and that it is uniformly bounded. 
The former condition is a reasonable assumption to make - clearly, if the prior does not place any mass in the neighborhood of the true parameter (assuming one exists) then neither will the posterior. The uniform boundedness condition on the other hand is attendant to a loss of generality. In particular, we cannot assume certain heavy-tailed priors (such as Pareto) which might be important for some engineering applications. Second, we also make the mild assumption that the likelihood function is locally asymptotically normal (LAN) in Assumption~\ref{assume:lan}. This is a standard assumption that holds for a variety of statistical/stochastic models. However, while the LAN assumption will be critical for establishing the asymptotic consistency results, it is unclear if it is necessary as well. We observe that~\citet{WaBl2017} make a similar assumption in analyzing the consistency of \klvb. We note that any model $P_\theta$ that is twice differentiable in the parameter $\theta$ satisfies the LAN condition~\citep{vdV00}. 
Also critical to the consistency result are the properties of the variational family. Assumption~\ref{assume:var} is a mild condition that insists on there existing Dirac delta distributions in an open neighborhood of the true parameter $\theta_0$. This is usually easy to verify: if the variational family consists of Gaussian distributions, for instance, then Dirac delta distributions are present at all points in the parameter space. 
Next, we assume that the variational family contains `good sequences' that are constructed so as to converge at the same rate as the true posterior (in sequence with the sample size), with the first moment of an element in the sequence the maximum likelihood estimator of the parameter (at a given sample size). We also require the tails of the good sequence to bound the tails of the true posterior. We provide examples that verify the existence of good sequences in commonly used variational families, such as the mean-field family.

The proof of Theorem~\ref{thm:consistency} is a consequence of a series of auxiliary results. First, in Lemma~\ref{lem:ndegen} we characterize $\alpha$-R\'enyi minimizers and show that the sequence must have a  Dirac delta distribution at the true parameter $\theta_0$ in the large sample limit.\ Then, in Lemma~\ref{lem:subopt1} we argue that any convex combination of a Dirac delta distribution at the true parameter $\theta_0$ with any other distribution can not achieve zero $\alpha$-R\'enyi divergence in the limit. \ Next, we show in Proposition~\ref{prop:UBfin} that the $\alpha$-R\'enyi divergence between the true posterior and the closest variational approximator is bounded above in the large sample limit.\ We demonstrate this by showing that a `good sequence' of distributions (see Assumption~\ref{def:gsequence}) has asymptotically bounded $\alpha$-R\'enyi divergence, implying that the minimizers do as well.\ Note that this  does not yet prove that the minimizing sequence converges to a Dirac delta distribution at $\theta_0$.
	
The next stage of the analysis is concerned with demonstrating that the minimizing sequence does indeed converge to a Dirac delta distribution concentrated at the true parameter. We demonstrate this fact as a consequence of Proposition~\ref{prop:UBfin}, Lemma~\ref{lem:ndegen}, and Lemma~\ref{lem:subopt1}. In essence, Theorem~\ref{thm:consistency} shows that, $\alpha$-R\'enyi minimizing distributions are arbitrarily close to a good sequence, in the sense of R\'enyi divergence with the posterior in the large sample limit.
	
In our next result in Theorem~\ref{thm:roc}, under additional regularity conditions, we further characterize the rate of convergence of the $\alpha-$R\'enyi minimizers. We demonstrate that the $\alpha-$R\'enyi minimizing sequence cannot concentrate to a point in the parameter space at a faster rate than the true posterior concentrates at the true parameter $\theta_0$. Consequently, the tail mass in the $\alpha$-R\'enyi minimizer could dominate that of the true posterior. This is in contrast with~\klvb, where the evidence lower bound (ELBO) maximizer typically under-estimates the variance of the true posterior. 

Here is a brief roadmap of the paper. In Section~\ref{sec:b-VB}, we formally introduce the $\alpha$-R\'enyi methodology, and rigorously state the necessary regularity assumptions. We present our main result in Section~\ref{sec:asymptote}, presenting only the proofs of the primary results. In Section~\ref{sec:asymptoteEP} we also recover the consistency of \blue{$1-$R\'enyi},  \red{idealized expectation propagation~(EP)} approximate posteriors, \blue{the global minimizer of EP objective} as a consequence of the results in~Section~\ref{sec:asymptote}. \blue{In Section~\ref{sec:latent}, we generalize the notion of good sequence to the models with local latent parameters and under some additional regularity conditions, 
     prove asymptotic consistency of the $\alpha$-R\'enyi approximate posterior over global latent parameters}. All proofs of auxiliary and technical results are delayed to the Appendix.


\section{Variational Approximation Using $\alpha-$R\'enyi Divergence}~\label{sec:b-VB}
We assume that the data-generating distribution is parametrized by $\theta \in \Theta \subseteq \bbR^d$, $d \geq 1$ and is absolutely continuous with respect to the Lebesgue measure, so that the likelihood function $p(\cdot|\theta)$ is well-defined. We place a prior $\pi(\theta)$ on the unknown $\theta$, and denote $\pi(\theta|\nX) \propto p(\theta,\nX)$ as the posterior distribution, where $\nX = \{\xi_1,\ldots, \xi_n\}$ are the $n$ independent and identically distributed (i.i.d.) observed samples generated from the `true' measure $P_{\theta_0}$ {\color{black}in the likelihood family}. 
In this paper we will study the $\alpha-$R\'enyi-approximate posterior $q^*_n$ that minimizes the $\alpha-$R\'enyi divergence between $\pi(\theta| \mathbf X_n)$ and $\tilde{q}(\cdot)$ in \blue{some} set $\sQ$ for a given $\alpha>1$; that is, 
\begin{eqnarray}
  q^*_n(\theta) := \text{argmin}_{\tilde{q} \in \mathcal{Q}} 
  \left\{D_{\alpha} \left( \pi(\theta| \mathbf X_n)\| \tilde{q}(\theta) \right):= \frac{1}{\alpha-1} \log \int_{\Theta} \tilde{q}(\theta)   \left(\frac{\pi(\theta| \mathbf X_n)}{\tilde{q}(\theta) } \right)^{\alpha} d\theta \right\} . \label{eq:ep_opt}
\end{eqnarray}

Recall that
\begin{definition}[Dominating distribution]~\label{def:dom}
	The distribution $Q$ dominates the distribution P ($P \ll Q$), when $P$ is absolutely continuous with respect to $Q$; that is, $supp(P) \subseteq supp(Q)$.
\end{definition}
Clearly, when $\alpha>1$, the $\alpha-$R\'enyi divergence in~\eqref{eq:ep_opt} is infinite for any distribution $q(\theta) \in \cQ$  that does not dominate the true posterior distribution \citep{TErvan2012}. Intuitively, this is the reason why the $\alpha$-R\'enyi approximation can better capture the spread of the posterior distribution.

Our goal is to study the statistical properties of the $\alpha-$R\'enyi-approximate posterior as defined in~\eqref{eq:ep_opt}. In particular, we show that under certain regularity conditions on the likelihood, the prior, and the variational family the $\alpha-$R\'enyi-approximate posterior is consistent or converges weakly to a Dirac delta distribution at the true parameter $\theta_0$ as the number of observations $n \to \infty.$ 
\blue{
\subsection{Asymptotic Notations}
We first define asymptotic notations that frequently appear in our proofs and assumptions. We write $a_n \sim b_n$ when the sequence $\{a_n\}$ can be approximated by a sequence $\{b_n\}$ for large $n$, 
so that the ratio $\frac{a_n}{b_n}$ approaches 1 as $n \to \infty$, $a_n = O(b_n)$ as $n \to \infty$, when there exists a positive number $M$ and $n_0\geq 1$, such that $a_n \leq M b_n \ \forall n \geq n_0$, and $a_n \lesssim b_n$ when the sequence $\{a_n\}$ is bounded above by a sequence $\{b_n\}$ for large $n$.  }

\subsection{Assumptions and Definitions}
First, we assume the following restrictions on permissible priors.
\vspace{.5em}
\begin{assumption}[Prior Density]~\label{assume:prior}
  \begin{enumerate}
    \item[(1)] The prior density function $\pi(\theta)$ is continuous with non-zero measure in the neighborhood of the true parameter
    $\theta_0$, and
    \item[(2)] there exists a constant $M_p > 0$ such that
    $\pi(\theta) \leq M_p~\forall \theta \in \Theta$ and $\bbE_{ \pi(\theta)}[|\theta|]< \infty$. 
  \end{enumerate}
\end{assumption}

Assumption~\ref{assume:prior}(1) is typical in Bayesian consistency analyses - quite obviously, if the prior does not place any mass around the true parameter then the (true) posterior will not either. Indeed, it is well known~\citep{Sc1965,Gh1997} that for any prior that satisfies Assumption~\ref{assume:prior}(1), under very mild assumptions, 
\begin{align}
\label{eq:4}
\pi(U | \nX) = \int_{U} \pi(\theta | \mathbf X_n) d\theta \rightarrow 1 \quad~P_{\theta_0}-a.s.~\text{as}~n\to\infty,
\end{align}
where $P_{\theta_0}$ represents the true data-generating distribution, $U$ is some neighborhood of the true parameter $\theta_0$.
Assumption~\ref{assume:prior}(2), on the other hand, is a mild technical condition which is satisfied by a large class of prior distributions, for instance, many of the exponential-family distributions. For simplicity, we write $q_n (\theta) \Rightarrow q(\theta)$ to represent weak convergence of the distributions corresponding to the densities $\{q_n\}$ and $q$.

We define a generic probabilistic order term, $o_{P_\theta}(1)$ with respect to measure $P_{\theta}$ as follows  
\vspace{0.5em}
\begin{definition}~\label{def:prob}
	 A sequence of random variables $\{\xi_n\}$ is of  probabilistic order $o_{P_\theta}(1)$ when
	\[ \lim_{n \to \infty} P_{\theta}( |\xi_n| > \delta) = 0, \text{ for any $\delta > 0$ }.\]
	\end{definition} 
\red{We write $a_n \sim b_n$ when the sequence $\{a_n\}$ can be approximated by a sequence $\{b_n\}$ for large $n$, so that the ratio $\frac{a_n}{b_n}$ approaches 1 as $n \to \infty$, $a_n = O(b_n)$ as $n \to \infty$, when there exists a positive number $M$ and $n_0\geq 1$, such that $a_n \leq M b_n \ \forall n \geq n_0$, and $a_n \lesssim b_n$ when the sequence $\{a_n\}$ is bounded above by a sequence $\{b_n\}$ for large $n$.} Next, we assume the likelihood function satisfies the following asymptotic normality property (see ~\cite{vdV00} as well),
\vspace{1em}
\begin{assumption}[Local Asymptotic Normality]~\label{assume:lan}
  Fix $\theta_0 \in \Theta$. The sequence of log-likelihood functions $\{ \log P_n(\theta) = \sum_{i=1}^n \log p(x_i|\theta) \}$ satisfies a \emph{local asymptotic normality (LAN)} condition, if there exists a sequence of matrices $\{r_n\}$, a matrix $I(\theta_0)$ and a sequence of random vectors $\{\Delta_{n,\theta_0}\}$ weakly converging to $\sN(0,I(\theta_0)^{-1})$ as $n \to \infty$, such that for every compact set $K \subset \mathbb{R}^d$ 
  \[
  \sup_{h \in K} \left| \log {P_n(\theta_0 + r_n^{-1} h)} - \log {P_n(\theta_0)} - h^T I(\theta_0)
  \Delta_{n,\theta_0} + \frac{1}{2} h^T I(\theta_0)h \right| \xrightarrow{P_{\theta_0}} 0 \ \text{as $n  \to \infty$  }.
  \]
\end{assumption}
The LAN condition is standard, and holds for a wide variety of models. The assumption affords significant flexibility in the analysis by allowing the likelihood to be asymptotically approximated by a scaled Gaussian centered around $\theta_0$~\citep{vdV00}. We observe that~\cite{WaBl2017} makes a similar assumption in their consistency analysis of the variational lower bound. All statistical models $P_{\theta}$, which are differentiable in quadratic mean with respect to parameter $\theta$, satisfy the LAN condition with $r_n = \sqrt{n} I$, where $I$ is an identity matrix~\cite[Chapter-7]{vdV00}. \blue{Also, all models $P_{\theta}$ which are twice continuously differentiable in $\theta$ are also differentiable in quadratic mean and thus satisfy LAN condition, for instance most exponential family models satisfy the LAN condition.  } 

 Now, let $\delta_{\theta}$ represent the Dirac delta, or singular distribution, concentrated at the parameter $\theta$.\begin{definition}[Degenerate distribution]\label{def:degen}
 A sequence of distributions $\{q_n(\theta)\}$ converges weakly to $\delta_{\theta'}$ that is, $q_n(\theta) \Rightarrow \delta_{\theta'} $ for some $ \theta' \in \Theta$,  if and only if  $\forall \eta > 0$
	\begin{align*}
		\lim_{n \to \infty } \int_{\{|\theta -\theta'| > \eta\}} {q}_n(\theta) 	d\theta = 0.
	\end{align*}
 \end{definition}
 

We use the term `non-degenerate' for a sequence of distributions that does not converge in distribution to a Dirac delta distribution. We also use the term `non-singular' to refer to a distribution that does not contain any singular components (i.e., it is absolutely continuous with respect to the Lebesgue measure). 
If a distribution contains both singularities and absolutely continuous components we term it a `singular distribution'. \blue{More formally,
 \begin{definition}[Singular distributions]
  Let $d(\theta)$ be a distribution with support $\Theta$ and for any $i\in\{1,\ldots,K\}$ and $K<\infty$ denote $\delta_{\theta_i},$ as the Dirac delta distributions at $\theta_i$ for any $
  \theta_i \in \Theta$, then we define singular distribution $q(\theta)$;
  \[q(\theta):= wd(\theta)+\sum_{i=1}^Kw^i\delta_{\theta_i},\]
  where $w,\{w^i\}_{i=1}^K \in [0,1)$ and $w+\sum_{i=1}^{K}w^i=1$ with at least one of the weights $\{w^i\}_{i=1}^K$ strictly positive.
     \end{definition}
}
Finally, we come to the conditions on the variational family $\cQ$. 
\vspace{0.25em}
\begin{assumption}[Variational Family]~\label{assume:var}
  The variational family $\cQ$ must contain all Dirac delta distributions in some open neighborhood of $\theta_0 \in \Theta$.
\end{assumption}  
Since we know that the posterior converges weakly to a Dirac delta distribution function, this assumption is a necessary condition to ensure that the variational approximator exists in the limit.  Next, we define the rate of convergence of a sequence of distributions to a Dirac delta distribution as follows.
\vspace{1em}
\begin{definition}[Rate of convergence]\label{def:roc}
	A sequence of distributions $\{q_n(\theta) \}$ converges weakly to $\delta_{\theta_1}$, $\forall \theta_1 \in\Theta$ at the rate of $\gamma_n$ if 
	\begin{enumerate}
		\item[(1)] the  sequence of means $ \{\check \theta_n := \int \theta q_n(\theta) d\theta \}$ converges to $\theta_1$ as $n\to \infty$, and 
		\item[(2)] the variance of $\{q_n(\theta) \}$ satisfies
		\[E_{q_n(\theta)}[|\theta - \check \theta_n|^2] = O\left (\frac{1}{\gamma_n^2} \right).\]
	\end{enumerate}
\end{definition}

A crucial assumption, on which rests the proof of our main result, is the existence of what we call a `good sequence' in $\cQ$.
\vspace{0.5em}
\begin{assumption}[\Good]\label{def:gsequence}
	\blue{For any $\bar M>0$}, the variational family $\cQ$ contains a sequence of distributions $\{\bar{q}_n(\theta)\}$ with the following properties:
	\begin{enumerate}
        \item[(1)] there exists $n_1 \geq 1$ such that $\int_{\Theta}\theta \bar q_n(\theta) d\theta = \hat{\theta}_n $, where $\hat{\theta}_n$ is the maximum likelihood estimate, for each $n\geq n_1$,
		\item[(2)] there exists $n_{\bar{M}}\geq 1$ such that the rate of convergence is $\gamma_n = \sqrt{n}$ , \blue{that is \( E_{\bar q_n(\theta)}[|\theta - \hat \theta_n|^2] \leq  \frac{\bar M}{\gamma_n^2}  \) for each $n\geq n_{\bar M}$,}
		\item[(3)] there exist a compact ball $K \subset  \Theta$ containing the true parameter $\theta_0$ and $n_2\geq1$, such that the sequence of Radon-Nikodym derivatives of the posterior density with respect to the sequence $\{\bar{q}_n\}$ exists and is bounded above by a finite positive constant $M_r$ outside of $K$ for all $n \geq n_2$: \[\frac{  \ \pi(  \theta | \nX)}{  \bar{q}_n(\theta)} \leq M_r, \ \forall \theta \in \Theta \backslash K \text{ and } \forall n\geq n_2, \quad P_{\theta_0}-a.s. \]
		\item[(4)] there exists $n_3\geq 1$ such that the good sequence $\{\bar q_n(\theta)\}$ is log-concave in $\theta$ for all $n \geq n_3$.
		
	\end{enumerate} 
We term such a  sequence of distributions as `good  sequences'. 
\end{assumption}

The first two parts of the assumption hold so long as the variational family $\mathcal{Q}$ contains an open neighborhood of  distributions around $\delta_{\theta_0}$. The third part essentially requires that for $n \geq n_2$, the tails of $\{\bar{q}_n(\theta)\}$ must decay no faster than the tails of the posterior distribution. Since, the good sequence converges weakly to $\delta_{\theta_0}$, this assumption is a mild technical condition. The last assumption implies that the good sequence is, for large sample sizes, a maximum entropy distribution under some deviation constraints on the entropy maximization problem~\citep{grechuk2009maximum}. Note that this does not imply that the good sequence is necessarily Gaussian (which is the maximum entropy distribution specifically under standard deviation constraints). 

We note that this assumption is on the family $\mathcal{Q}$, and not on the minimizer of the R\'enyi divergence.
We demonstrate the existence of good sequences for some example models. 
\vspace{1em}
\begin{example}
Consider a model whose likelihood is an $m$-dimensional multivariate Gaussian likelihood with unknown mean vector $\pmb{\mu}$ and known covariance matrix $\mathbf{\Sigma}$. Using an $m$-dimensional multivariate normal distribution with mean vector $\pmb{\mu_0}$ and covariance matrix $\mathbf{\Sigma}$ as conjugate prior, the posterior distribution is
\[\pi(\pmb\mu|\nX) = \sqrt{ \frac{(n+1)^m}{  (2\pi)^m det\left( {\mathbf \Sigma} \right) } } e^{-\frac{n+1}{2} \left( \pmb{\mu} - \frac{\sum_{i=1}^{n}X_i + \pmb{\mu_0}}{n +1} \right)^T \mathbf\Sigma^{-1} \left( \pmb{\mu} - \frac{\sum_{i=1}^{n}X_i + \pmb{\mu_0}}{n +1} \right) }, \] 
where exponents `$T$' and `$-1$' denote transpose and inverse. Next, consider the mean-field variational family, that is the product of $m$ 1-dimensional normal distributions. Consider a sequence in the variational family with mean $\{\mu^j_{q_n}, j \in \{1,2,\ldots, m\} \}$ and variance $\bigg\{\frac{\sigma^2_j}{\gamma^2_n}, j \in \{1,2,$ $\ldots, m\}  \bigg\}$:
\[ q_n(\pmb{\mu})= \prod_{j=1}^{m} \sqrt{\frac{\gamma^2_n}{2\pi \sigma^2_j  }} e^{-\frac{\gamma^2_n}{2 \sigma_j^2} \left( \mu_j - \mu_{q_n}^j  \right)^2  
} = \sqrt{\frac{\gamma^{2m}_n}{ (2\pi)^m det( \mathbf I_{\sigma})  }} e^{-\frac{\gamma^2_n}{2 } \left( \pmb \mu - \pmb \mu_{q_n}  \right)^T \mathbf I_{\sigma}^{-1}  \left( \pmb \mu - \pmb \mu_{q_n}  \right) 
} , \] 
where $\pmb \mu_{q_n} = \{\mu_{q_n}^1 , \mu_{q_n}^2,\ldots, \mu_{q_n}^m\}$ and $\mathbf I_{\sigma}$ is an $m\times m$ diagonal matrix with diagonal elements $\{ \sigma_1^2, \sigma_2^2,$
$\ldots, \sigma_m^2 \}$. Notice that $\gamma_n$ is the rate at which the sequence $\{q_n(\pmb{\mu})\}$ converges weakly. 
It is straightforward to observe that the variational family contains sequences that satisfy properties (1) and (2) in Assumption~\ref{def:gsequence}, that is 
\[ \gamma_n = \sqrt{n} \text{ and } \pmb{\mu_{q_n}} = \frac{\sum_{i=1}^{n}X_i + \pmb{\mu_0}}{n +1}. \]

For brevity, denote $\pmb{\tilde\mu}_n := \pmb \mu - \pmb \mu_{q_n} = \pmb{\mu} - \frac{\sum_{i=1}^{n}X_i + \pmb{\mu_0}}{n +1} $ . To verify property (3) in Assumption~\ref{def:gsequence} consider the ratio,
\begin{align*}
	\frac{\pi(\pmb\mu|\nX)}{q_n(\pmb{\mu})} &= \frac{\sqrt{ \frac{(n+1)^m}{  (2\pi)^m det\left( {\mathbf \Sigma} \right) } } e^{-\frac{n+1}{2} \pmb{\tilde\mu}_n^T \mathbf\Sigma^{-1} \pmb{\tilde\mu}_n }}{\sqrt{\frac{\gamma^{2m}_n}{ (2\pi)^m det( \mathbf I_{\sigma})  }} e^{-\frac{\gamma^2_n}{2 } \pmb{\tilde\mu}_n^T \mathbf I_{\sigma}^{-1}  \pmb{\tilde\mu}_n 
	}}. 
\end{align*}
Using the fact that $\gamma^2_n = n < n+1$, $\frac{n+1}{\gamma^2_n}=1+\frac{1}{n}<2$, therefore the ratio above can be bounded above by
\begin{align*}
	\frac{\pi(\pmb\mu|\nX)}{q_n(\pmb{\mu})} &\leq  \sqrt{ \frac{ 2^m det( \mathbf I_{\sigma})}{  det\left( {\mathbf \Sigma} \right) } } \frac{ e^{-\frac{n+1}{2} \pmb{\tilde\mu}_n^T \mathbf\Sigma^{-1} \pmb{\tilde\mu}_n }}{ e^{-\frac{n+1}{2 } \pmb{\tilde\mu}_n^T \mathbf I_{\sigma}^{-1}  \pmb{\tilde\mu}_n 
	}} =  \sqrt{ \frac{ 2^m det( \mathbf I_{\sigma})}{  det\left( {\mathbf \Sigma} \right) } }  e^{-\frac{n+1}{2} \pmb{\tilde\mu}_n^T  \left( \mathbf\Sigma^{-1}  - \mathbf I_{\sigma}^{-1} \right) \pmb{\tilde\mu}_n }. 
\end{align*}

Observe that if the matrix $\left( \mathbf\Sigma^{-1}  - \mathbf I_{\sigma}^{-1} \right)$ is positive definite then the ratio above is bounded by $\sqrt\frac{2^mdet( \mathbf I_{\sigma})}{  det\left( {\mathbf \Sigma} \right) }$ and if $\cQ$ is large enough it will contain  distributions that satisfy this condition. To fix the idea, consider the univariate case, where the positive definiteness implies that the variance of the good sequence is greater than the variance of the posterior for all large enough `$n$'. That is, the tails of the good sequence decay slower than the tails of the posterior.
\end{example}
\vspace{1em}
\begin{example}
Consider a model whose likelihood is a univariate normal distribution with unknown mean $\mu$ and known variance $\sigma$. Using a univariate normal distribution with the mean $\mu_0$ and the variance $\sigma$ as prior, the posterior distribution is
\begin{align}
\pi(\mu|\nX) =  \sqrt{\frac{n+1}{2\pi \sigma^2}} e^{  -\frac{(n+1)} {2 \sigma^2} \left( \mu - \frac{ \mu_0 + \sum_{i=1}^{n}X_i }{n+1} \right)^2 }.  
\label{eq:eqg1}
\end{align} 
Next, suppose the variational family $\cQ$ is the set of all Laplace distributions. Consider a sequence $\{q_n(\mu)\}$ in $\cQ$ with the location and the scale parameter $k_n$ and $b_n$ respectively, that is
\[q_n(\mu) = \frac{1}{2  b_n} e^{-\frac{|\mu- k_n |}{b_n}}. \]
To satisfy properties (1) and (2) in Assumption~\ref{def:gsequence}, we can choose $k_n= \frac{ \mu_0 + \sum_{i=1}^{n}X_i }{n+1} $ and $b_n= \sqrt{\frac{\pi \a^{\frac{1}{\a-1}}\sigma^2}{2n}} ,\ \forall \a>1$. For brevity denote  $\tilde\mu_n = \mu - \frac{ \mu_0 + \sum_{i=1}^{n}X_i }{n+1}$. To verify property (3) in Assumption~\ref{def:gsequence} consider the ratio,
\begin{align*}
\frac{\pi(\mu|\nX)}{q_n(\mu)} &= \frac{\sqrt{\frac{n+1}{2\pi \sigma^2}} e^{  -\frac{(n+1)} {2\sigma^2} \tilde\mu_n^2 } }{ \frac{1}{2} \sqrt{\frac{2n}{\pi \a^{\frac{1}{\a-1}}\sigma^2}} e^{-  \frac{\sqrt{2n}|\tilde\mu_n |}{\sqrt{{\pi \a^{\frac{1}{\a-1}}\sigma^2}}}  }} \leq \sqrt{\frac{2}{\a^{\frac{1}{\a-1}}}} \frac{ e^{  -\frac{(n+1)} {\pi \a^{\frac{1}{\a-1}}\sigma^2} \tilde\mu_n^2 } }{  e^{- \left | \frac{\sqrt{2(n+1)}|\tilde\mu_n |}{\sqrt{{\pi \a^{\frac{1}{\a-1}}\sigma^2}}} \right| }}  \leq \sqrt{\frac{2}{\a^{\frac{1}{\a-1}}}} e^{1/2},
\end{align*}
where the  last inequality follows due  to the fact that $ e^{-(\frac{x^2}{2}-|x|)}<e^{1/2}$. 
%

For the same posterior, we can also choose $\cQ$ to be the set of all Logistic distributions. Consider a sequence $\{q_n(\mu)\}$ in this variational family with the mean and the scale parameter $m_n$ and $s_n$ respectively; that is
\[ q_n(\mu) = \frac{1}{s_n}\left( e^{\frac{\mu-m_n}{2s_n}} + e^{-\frac{\mu-m_n}{2s_n}} \right)^{-2}. \] To satisfy properties (1) and (2) in Assumption~\ref{def:gsequence}, we can choose $m_n= \frac{ \mu_0 + \sum_{i=1}^{n}X_i }{n+1} $ and $s_n= \sqrt{\frac{2\pi \a^{\frac{1}{\a-1}}\sigma^2}{n+1}} ,\ \forall \a>1$. For brevity denote  $\tilde\mu_n = \mu - \frac{ \mu_0 + \sum_{i=1}^{n}X_i }{n+1}$. To verify property (3) in Assumption~\ref{def:gsequence} observe that,
\begin{align*}
\frac{\pi(\lambda|\nX)}{q_n(\lambda)} &= \frac{\sqrt{\frac{n+1}{2\pi \sigma^2}} e^{  -\frac{(n+1)} {2\sigma^2} \left( \mu - \frac{ \mu_0 + \sum_{i=1}^{n}X_i }{n+1} \right)^2 } }{\frac{1}{s_n}\left( e^{\frac{\mu-m_n}{2s_n}} + e^{-\frac{\mu-m_n}{2s_n}} \right)^{-2}} = \frac{1}{\sqrt{\a^{\frac{1}{\a-1}}}} e^{  -\left(\frac{\tilde \mu_n} {s_n} \right)^2 } \left( e^{  \left(\frac{\tilde \mu_n} {2s_n} \right) }  + e^{  -\left(\frac{\tilde \mu_n} {2s_n} \right) } \right) \leq \frac{1}{\sqrt{\a^{\frac{1}{\a-1}}}} 2e^{1/16},
\end{align*}
where the  last inequality follows due  to the fact that $ e^{-x^2}\left( e^{x/2} + e^{-x/2}  \right) < 2e^{1/16}$. 
%

\end{example}
\vspace{1em}
\begin{example}
Consider a univariate exponential likelihood model with the unknown rate parameter $\lambda$. For some prior distribution $\pi(\lambda)$, the posterior distribution is
\[\pi(\lambda|\nX) = \frac{\pi(\lambda) \lambda^n e^{-\lambda \sum_{i=1}^n X_i }}{\int \pi(\lambda) \lambda^n e^{-\lambda \sum_{i=1}^n X_i } d\lambda} .  \] Choose $\cQ$ to be the set of Gamma distributions. Consider a sequence $\{q_n(\mu)\}$ in the variational family with the shape and the rate parameter $k_n$ and $\beta_n$ respectively, that is
\[q_n(\lambda) = \frac{\beta_n^{k_n}}{\Gamma(k_n)} \lambda^{k_n-1}e^{-\lambda \beta_n}, \]
where $\Gamma(\cdot)$ is the $\Gamma-$ function. To satisfy properties (1) and (2) in Assumption~\ref{def:gsequence}, we can choose $k_n= n+1$ and $\beta_n= \sum_{i=1}^{n}X_i$. To verify property (3) in Assumption~\ref{def:gsequence} consider the ratio,
\begin{align*}
\frac{\pi(\lambda|\nX)}{q_n(\lambda)} = \frac{\pi(\lambda) \lambda^n e^{-\lambda \sum_{i=1}^n X_i } }{ \frac{\beta_n^{k_n}}{\Gamma(k_n)} \lambda^{k_n-1}e^{-\lambda \beta_n} \int \pi(\lambda) \lambda^n e^{-\lambda \sum_{i=1}^n X_i } d\lambda } = \frac{\pi(\lambda)\Gamma(n+1) }{\left( \sum_{i=1}^n X_i \right)^{n+1}\int \pi(\lambda) \lambda^n e^{-\lambda \sum_{i=1}^n X_i } d\lambda }. 
\end{align*}

Now, observe that $ \frac{ \left( \sum_{i=1}^n X_i \right)^{n+1} }{ \Gamma(n+1)} \lambda^n e^{-\lambda \sum_{i=1}^n X_i } $ is the density of Gamma distribution with the mean $\frac{n+1}{\sum_{i=1}^n X_i}$ and the variance $\frac{1}{n+1}\left( \frac{n+1}{\sum_{i=1}^n X_i} \right)^2$. Since, we assumed in Assumption~\ref{assume:prior}(2) that $\pi(\lambda)$ is bounded from above by $M_p$, therefore for large $n$, 
\(\frac{ \left( \sum_{i=1}^n X_i \right)^{n+1} }{ \Gamma(n+1)} \int \pi(\lambda) \lambda^n e^{-\lambda \sum_{i=1}^n X_i } d\lambda \sim  \pi\left( \frac{n+1}{\sum_{i=1}^n X_i} \right) \). Hence, it follows that for large enough $n$
\[\frac{\pi(\lambda|\nX)}{q_n(\lambda)} \leq \frac{M_p}{\pi(\lambda_0)},\]
where $\frac{\sum_{i=1}^n X_i}{n+1} \to \frac{1}{\lambda_0}$ as $n \to \infty$.

\end{example}


\section{Consistency of $\alpha-$R\'enyi Approximate Posterior}~\label{sec:asymptote}
Recall that the $\alpha-$R\'enyi-approximate posterior $q^*_n$ is defined as
\begin{eqnarray}
q^*_n(\theta) := \text{argmin}_{\tilde{q} \in \mathcal{Q}} 
\left\{D_{\alpha} \left( \pi(\theta| \mathbf X_n)\| \tilde{q}(\theta) \right):= \frac{1}{\alpha-1} \log \int_{\Theta} \tilde{q}(\theta)   \left(\frac{\pi(\theta| \mathbf X_n)}{\tilde{q}(\theta) } \right)^{\alpha} d\theta \right\} . \label{eq:ep_opt1}
\end{eqnarray}
We now show that under the assumptions in the previous section, the $\alpha-$R\'enyi approximators are asymptotically consistent as the sample size increases in the sense that \(q_n^* \Rightarrow \delta_{\theta_0} \text{ in -} P_{\theta_0} \) probability as $n \to \infty$. To illustrate the ideas clearly, we present our analysis assuming a univariate parameter space, and that the model $P_{\theta}$ is twice differentiable in parameter $\theta$, and therefore satisfies the LAN condition with $r_n = \sqrt{n}$ \citep{vdV00}. The LAN condition together with the existence of a sequence of test functions \cite[Theorem 10.1]{vdV00} also implies that the posterior distribution converges weakly to $\delta_{\theta_0}$ at the rate of $\sqrt{n}$. The analysis can be easily adapted to multivariate parameter spaces. 

We will first establish some structural properties of the minimizing sequence of distributions. We show that for any sequence of distributions converging weakly to a non-singular distribution the $\alpha-$R\'enyi divergence is unbounded in the limit. 

\vspace{1em}
\begin{lemma}\label{lem:ndegen}
	Under Assumptions~\ref{assume:prior},~\ref{assume:lan},~\ref{assume:var}, and \ref{def:gsequence}, the $\alpha-$R\'enyi divergence between the true posterior and the sequence $\{q_n(\theta)\} \subset \cQ$ can only be finite in the limit if $q_n(\theta)$ converges weakly to a singular distribution $q(\theta)$ with a Dirac delta distribution at the true parameter $\theta_0$.
\end{lemma}
\vspace{0em}

The result above implies that the $\alpha-$R\'enyi approximate posterior must have a Dirac delta distribution component at $\theta_0$ in the limit; that is, it should converge in distribution to $\delta_{\theta_0}$ or a convex combination of $\delta_{\theta_0}$ with singular or non-singular distributions as $n \to \infty$. Next, we consider a sequence $\{q'_n(\theta)\} \subset \cQ$ that  converges weakly to a convex combination of $\delta_{\theta_0}$ and singular or non-singular distributions $q_i(\theta),~i \in \{1,2,\ldots \}$ such that for weights $\{w^i \in (0,1) : \sum_{i=1}^{\infty} w^i=1\},$
\begin{align}
q'_n(\theta) \Rightarrow w^j\delta_{\theta_0} +  \sum_{i=1, i\neq j}^{\infty} w^i q_i(\theta).
\label{eq:eq14i}
\end{align}
In the following result, we show that the $\alpha-$R\'enyi divergence between the true posterior and the sequence $\{q'_n(\theta)\}$ is bounded below by a positive number. 
\vspace{1em}
\begin{lemma}
	\label{lem:subopt1}
	Under Assumption~\ref{assume:prior}, the $\alpha-$R\'enyi divergence between the true posterior and the sequence $\{q'_n(\theta) \in \cQ\}$ is bounded away from zero; that is
	\[ \liminf_{n \to \infty} D_\alpha(\pi(\theta|\nX)\| q'_n(\theta)) \geq \eta > 0\quad P_{\theta_0}-a.s. \] 
\end{lemma}  

We also show in Lemma~\ref{lem:subopt} in the appendix that if in~\eqref{eq:eq14i} the components $\{q_i(\theta) ~i \in \{1,2,\ldots \}\}$ are singular, then
with $w^j$ is the weight of $\delta_{\theta_0}$, we have 
\[ \liminf_{n \to \infty} D_\alpha(\pi(\theta|\nX)\| q'_n(\theta)) \geq 2 (1-w^j)^2 >0  \quad P_{\theta_0}-a.s. \]

A consistent sequence asymptotically achieves zero $\alpha-$R\'enyi divergence. To show its existence, we first provide an asymptotic upper-bound on the minimal $\alpha-$R\'enyi divergence in the next proposition. This, coupled with the previous two structural results, will allow us to prove the consistency of the minimizing sequence.
\vspace{.5em}
\begin{proposition}\label{prop:UBfin}
For a given $\a>1$ and under Assumptions~\ref{assume:prior},~\ref{assume:lan},~\ref{assume:var}, and~\ref{def:gsequence},
     for any~\good~$\bar q_n(\theta)$ there exist $n_0\geq 1$ \red{, $n_M \geq 1$,} and $\bar M > 0$ such that \red{$\bbE_{\bar q_n(\theta)}[|\theta - \hat \theta_n|^2] \leq \frac{\bar M}{n}$ for all $n \geq n_M$ and $\bar M I(\theta_0) \geq \frac{ \alpha^{\frac{1}{\a-1}}}{\bar e}$ } for all $n \geq n_0$, 	
    the minimal $\alpha-$R\'enyi divergence satisfies  
	\begin{align}
     \blue{\min_{q \in \mathcal Q}  D_{\alpha}(\pi(\theta|\nX) \|q(\theta)) \leq D_{\alpha}(\pi(\theta|\nX) \|\bar q_n(\theta)) \leq B = \frac{1}{2}\log\left(\frac{\bar e \bar M I(\theta_0)}{\alpha^{\frac{1}{\alpha-1}}}\right)+ o_{P_{\theta_0}}(1) ,} 
	\end{align}
    where $I(\theta_0)$  is defined in Assumption~\ref{assume:lan} and $\bar e$ is the Euler's constant.
\end{proposition}

Now Proposition~\ref{prop:UBfin}, Lemma~\ref{lem:ndegen}, and Lemma~\ref{lem:subopt1} allow us to prove our main result that the $\alpha-$R\'enyi approximate posterior converges weakly to $\delta_{\theta_0}$.    
\vspace{0.5em} 
\begin{theorem}\label{thm:consistency}
	Under Assumptions~\ref{assume:prior},~\ref{assume:lan},~\ref{assume:var}, and~\ref{def:gsequence}, the $\alpha-$R\'enyi approximate posterior $q^*_n(\theta)$ converges weakly to a Dirac delta  distribution at the true parameter $\theta_0$; that is, 
	\[\blue{q^*_n  \Rightarrow \delta_{\theta_0}\ ~\text{in-}P_{\theta_0} \text{ probability as $n \to \infty$.} }\]
\end{theorem}
\vspace{-3em}
\blue{
    \begin{proof}
        First, we argue that there always exists a sequence $\{\tilde{q}_n(\theta)\} \subset \cQ$ such that for every $\eta>0$
        \[ \lim_{n \to \infty } P_{\theta_0} \left( D_\alpha(\pi(\theta|\nX)\| \tilde{q}_n(\theta)) \leq \eta \right)=1. \]
        We demonstrate the existence of $\tilde{q}_n(\theta)$ by construction. Recall from Proposition~\ref{prop:UBfin}(2) that there exist $ 0 < \bar M < \infty$ and $n_0\geq 1$, such that for all $n \geq n_0$  
        \begin{align*}
        D_{\alpha}(\pi(\theta|\nX)  \|\bar  q_n(\theta)) &\leq  \frac{1}{2} \log \frac{\bar e \bar M I(\theta_0)}{ \a^{\frac{1}{\a-1}} }\blue{+ o_{P_{\theta_0}}(1)},
        \end{align*}
        where  $\bar  q_n(\theta)$ is the \good\ as defined in Assumption~\ref{def:gsequence} and $\bar e$ is the Euler's constant. 
        Now using the definition of $o_{P_{\theta_0}}(1)$, for every $\eta>0$, it follows from the inequality above that
        \begin{align}
        \lim_{n \to \infty } P_{\theta_0} \left( D_{\alpha}(\pi(\theta|\nX)  \|\bar  q_n(\theta)) -  \frac{1}{2} \log \frac{\bar e \bar M I(\theta_0)}{ \a^{\frac{1}{\a-1}}}   > \eta \right) \leq \lim_{n \to \infty } P_{\theta_0} \left( o_{P_{\theta_0}}(1) > \eta \right) =0.
        \end{align}
        Now a specific good sequence can be chosen by fixing \( \bar M = \tilde M :=  \frac{\a^{\frac{1}{\a-1}}}{\bar e I(\theta_0)}  \), implying that
        \begin{align}
        \lim_{n \to \infty } P_{\theta_0} \left( D_{\alpha}(\pi(\theta|\nX)  \|\tilde  q_n(\theta))   > \eta \right)  = 0.
        \label{eq:eqthm1}
        \end{align}
        The above result implies that there exist a sequence in family $\cQ$ such that $D_{\alpha}(\pi(\theta|\nX)  \|\tilde  q_n(\theta)) \to 0$ in $P_{\theta_0}$ -probability.\\  
        Next, we will show that the minimizing sequence must converge to a Dirac delta distribution in probability. The previous result shows that the minimizing sequence must have zero $\alpha$-R\'enyi divergence in the limit. Lemma~\ref{lem:ndegen} shows that the minimizing sequence must have a delta at $\theta_0$, since otherwise the $\alpha$-R\'enyi divergence is unbounded. Similarly, Lemma~\ref{lem:subopt1} shows that it cannot be a mixture of such a delta with other components, since otherwise the $\alpha$-R\'enyi divergence is bounded away from zero.  
        \\ 
        Therefore, it follows that the $\alpha-$R\'enyi approximate posterior $q^*_n(\theta)$ must converge weakly to a Dirac delta distribution at the true parameter $\theta_0,$ in $-P_{\theta_0}$ probability, thereby completing the proof. 
    \end{proof}
}

Note that the choice of $\bar M$ in the proof essentially determines the variance of the good sequence. As noted before, the asymptotic log-concavity of the good sequence implies that it is eventually an entropy maximizing sequence of distributions~\citep{grechuk2009maximum}. It does not necessarily follow that the sequence is Gaussian, however. If such a choice can be made (i.e., the variational family contains Gaussian distributions) then the choice of good sequence amounts to matching the entropy of a Gaussian distribution with variance $\frac{\alpha^{\frac{1}{\alpha-1}}}{\bar e I(\theta_0)}$. 

 We further characterize the rate of convergence of the $\alpha-$R\'enyi approximate posterior under additional regularity conditions. In particular, we establish an upper bound on the rate of convergence of the possible candidate $\alpha-$R\'enyi approximators \blue{ when the variational family is sub-Gaussian.  Additionally, we require that the posterior distribution satisfies the Bernstein-von Mises Theorem, that is for any compact set $K$ containing $\theta_0$
     \begin{align} 
     \int_{K} \pi(\theta|\nX) d\theta  =   \int_{K} \mathcal{N}(\theta;\hat \theta_n, (n I(\theta_0))^{-1}) d\theta + o_{P_{\theta_0}}(1). 
     \label{eq:BVM}
     \end{align}
      According to Theorem 10.1 in~\cite{vdV00}, the Bernstein-von Mises Theorem holds under Assumption~\ref{assume:prior},~\ref{assume:lan}, and the following additional assumption on the existence of consistent test functions:
 \vspace{1em}
 \begin{assumption}[Consistent Tests]\label{assume:test}
     For every $\e>0$ there exists a sequence  of tests $\phi_n(\nX)$ such that $i) \lim_{n \to \infty } \bbE_{P_{\theta_0}}(\phi_n(\nX))= 0$, and $ \lim_{n \to \infty } \sup_{\|\theta-\theta_0\|\geq  \e} \bbE_{P_{\theta_0}}(1-\phi_n(\nX)) = 0$.
     \end{assumption}
 } \red{The LAN condition with the existence of test functions \cite[Theorem 10.1]{vdV00} guarantees that the Bernstein-von Mises Theorem holds for the posterior distribution.} A further modeling assumption is  to choose a \blue{sub-Gaussian} variational family $\cQ$ that limits the variance.  \blue{We choose a sub-Gaussian sequence of distributions $\{q_n(\theta)\} \subset \cQ$}, that is for some positive constant $B$ and any $t\in \mathbb{R}$,
\begin{align}
 \bbE_{  q_n(\theta)} [e^{t\theta}] \leq e^{\tilde{\theta}_n t + \frac{B}{2\gamma_n^2} t^2 },
 \label{eq:SGseq}
 \end{align}
where $\tilde\theta_n$ is the mean of $q_n(\theta)$ and $\gamma_n$ is the rate (see Definition~\ref{def:roc}) at which $q_n(\theta)$ converges weakly to a Dirac delta distribution as $n \to \infty$.
\vspace{1em}
\begin{lemma}\label{thm:degen}
	Consider a sequence of sub-Gaussian distributions $\{q_n(\theta)\} \subset \cQ$, with parameters $B$ and $t$, that converges weakly to some Dirac delta distribution faster than the posterior converges weakly to $\delta_{\theta_0}$ (that is, $\gamma_n> \sqrt{n}$),
	and suppose the true posterior distribution satisfies the Bernstein-von Mises Theorem~\eqref{eq:BVM}. Then, there exists an $n_0 \geq 1$ such that the $\alpha-$R\'enyi divergence
	$D_{\a}(\pi(\theta|\nX)\| q_n(\theta)) $
	is infinite for all $n > n_0$.
\end{lemma} 

\blue{We use the above result to show that, when the variational family $\cQ$ is sub-Gaussian, then the $\alpha-$R\'enyi appropriate posterior cannot converge at a rate $\gamma_n$ faster than $\sqrt{n}$, that is the rate at which the posterior converges weakly to $\delta_{\theta_0}$.
\vspace{1em}
\begin{theorem}\label{thm:roc}
    Under Assumptions~\ref{assume:prior},~\ref{assume:lan},~\ref{assume:var},~\ref{def:gsequence}, and \ref{assume:test}, and $\cQ$ is a family of sub-Gaussian distribution, then the rate of convergence, $\gamma_n$, of $\alpha-$R\'enyi approximate posterior is bounded above by $\sqrt{n}$, that is $\gamma_n \leq \sqrt{n}$.  
    \end{theorem}
\begin{proof}
Since we choose the variational family to be sub-Gaussian, the $\alpha-$R\'enyi approximate posterior must be one of the sequences satisfying~\eqref{eq:SGseq} and as a consequence of Theorem~\ref{thm:consistency}, $\tilde\theta_n$ must converge to $\theta_0$ as $n\to \infty$. On the other hand, using Lemma~\ref{thm:degen}, it follows that the rate of convergence $\gamma_n$ of $\alpha-$R\'enyi approximate posterior  must be bounded above by $\sqrt{n}$, that is $\gamma_n \leq  \sqrt{n}$.    
    \end{proof}
}

\section{Consistency of \red{Idealized EP-} \blue{$\alpha$-R\'enyi} Approximate Posterior \blue{as $\alpha \to 1$ }}~\label{sec:asymptoteEP}
Our results on the consistency of $\alpha$-R\'enyi variational approximators in Section~\ref{sec:asymptote} \red{imply} \blue{can be a step forward in understanding }the consistency of posterior approximations obtained using expectation propogation (EP) \citep{Mi2001a,Mi2001b}. Observe that for any $n\geq1$, as $\alpha \to 1$,
\begin{eqnarray}
{D_{\alpha}} \left( \pi(\theta| \mathbf X_n)\| \tilde{q}(\theta) \right) \to \scKL \left( \pi(\theta| \mathbf X_n)\| \tilde{q}(\theta) \right),
\end{eqnarray}
where the limit is the EP objective using \scKL~divergence. We define the \red{EP} \blue{1-R\'enyi}-approximate posterior $s^*_n$ as the distribution in  the variational family $\cQ$ that minimizes the \scKL\ divergence between $\pi(\theta|\nX)$ and $\tilde{s}(\theta)$, where $\tilde{s}(\theta)$ is an element of $\cQ$:
\begin{eqnarray}
s^*_n(\theta) := \text{argmin}_{\tilde{s} \in \mathcal{Q}} 
\left\{\scKL \left( \pi(\theta| \mathbf X_n)\| \tilde{s}(\theta) \right):=  \int_{\Theta} \pi(\theta| \mathbf X_n)   \log\left(\frac{\pi(\theta| \mathbf X_n)}{\tilde{s}(\theta) } \right) d\theta \right\} . \label{eq:ep_opt2}
\end{eqnarray}
We note that the EP algorithm~\citep{Mi2001a} is a message-passing algorithm that optimizes an approximations to this objective~\citep{wainwright2008graphical}. Nevertheless, understanding this idealized objective is an important step towards understanding the actual EP algorithm. Furthermore, ideas from~\cite{LiTu2016} can be used to construct alternate algorithms that directly minimize~\eqref{eq:ep_opt2}. We thus focus on this objective, and show that under the assumptions in Section~\ref{sec:b-VB}, the \red{EP} \blue{1-R\'enyi}-approximate posterior is asymptotically consistent as the sample size increases, in the  sense that \(s_n^* \Rightarrow \delta_{\theta_0}, \ \text{ in-}P_{\theta_0} \) probability as $n \to \infty$. The proofs in this section are corollaries of the results in the previous section. 

Recall that the KL divergence lower-bounds the $\alpha-$R\'enyi divergence when $\alpha>1$; that is
\begin{align}
	 \scKL \left( p(\theta)\| q(\theta) \right) \leq D_{\alpha}\left( p(\theta)\| q(\theta) \right). 
	\label{eq:DKL}
	\end{align}
This is a direct consequence of Jensen's inequality. Analogous to Proposition~\ref{prop:UBfin}, we first show that the minimal \scKL \ divergence between the true Bayesian posterior and the variational family $\cQ$ is asymptotically bounded.

\vspace{.5em}
\begin{proposition}\label{prop:UBfinEP}
	For a given $\a>1$, and under Assumptions~\ref{assume:prior},~\ref{assume:lan},~\ref{assume:var},~\ref{def:gsequence}, and for any~\good~$\bar q_n(\theta)$ there exist $n_0\geq 1$ and $\bar M > 0$ such that the minimal \scKL \  divergence satisfies  
    \blue{
	\begin{align}
	 \min_{\tilde s \in \mathcal Q} \scKL \left( \pi(\theta| \mathbf X_n)\| \tilde{s}(\theta) \right) < B = \frac{1}{2}\log\left(\frac{\bar e \bar M I(\theta_0)}{\alpha^{\frac{1}{\alpha-1}}}\right)+ o_{P_{\theta_0}}(1) .
	\end{align}
    where $I(\theta_0)$  is defined in Assumption~\ref{assume:lan} and $\bar e$ is the Euler's constant.
}
\end{proposition}
\begin{proof}
	The result follows immediately from   Proposition~\ref{prop:UBfin} and~\eqref{eq:DKL}, since for any $\tilde s(\theta) \in \cQ$ and $\a > 1$,
\begin{align*}
\scKL \left( \pi(\theta| \mathbf X_n)\| \tilde{s}(\theta) \right) \leq    D_{\alpha}\left( \pi(\theta| \mathbf X_n)\| \tilde{s}(\theta) \right).
\end{align*}	
	\end{proof}
Next, we demonstrate that any sequence of distributions $\{s_n(\theta)\} \subset \cQ$ that converges weakly to a distribution $s(\theta) \in \sQ$ with positive probability outside the true parameter $\theta_0$ cannot achieve zero \scKL\ divergence in the limit. Observe that this result is weaker than Lemma~\ref{lem:ndegen}, and does not show that the KL divergence is necessarily infinite in the limit. This loses some structural insight. 
\vspace{0.5em}
\begin{lemma}
	\label{lem:subopt1ep}
	There exists an $\eta>0$ in the extended real line such that the \scKL\ divergence between the true posterior and sequence $\{s_n(\theta)\}$ is bounded away from zero; that is,
	\[ \liminf_{n \to \infty} \scKL(\pi(\theta|\nX)\| s_n(\theta)) \geq   \eta >0 \quad P_{\theta_0}-a.s. \] 
\end{lemma}  

Now using Proposition~\ref{prop:UBfinEP} and Lemma~\ref{lem:subopt1ep} we  show that the \red{EP} \blue{1-R\'enyi}-approximate posterior converges weakly to the $\delta_{\theta_0}$. 
\vspace{0.5em} 
\begin{theorem}\label{thm:consistencyEP}
	Under Assumptions~\ref{assume:prior},~\ref{assume:lan},~\ref{assume:var}, and~\ref{def:gsequence}, the \red{EP} \blue{1-R\'enyi}-approximate posterior $s^*_n(\theta)$ satisfies\blue{
	\[s^*_n  \Rightarrow \delta_{\theta_0} \quad ~\text{in-}P_{\theta_0} \text{ probability as $n \to \infty$}.\]}
\end{theorem}
\vspace{-2.5em}
\blue{
\begin{proof}
	Recall~\eqref{eq:eqthm1} from the proof of Theorem~\ref{thm:consistency} that there exists a good sequence $\tilde{q}_n(\theta)$, such that
	\begin{align*}
	 D_{\alpha}(\pi(\theta|\nX)  \|\tilde q_n(\theta)) \to  0 ~\text{in-}P_{\theta_0} \text{ probability as $n \to \infty$}. 
	\label{eq:eqthm2}
	\end{align*}
	Since the \scKL\ divergence is always non-negative, using~\eqref{eq:DKL} it follows that
	\begin{align}
	\nonumber
    \scKL (\pi(\theta|\nX)  \|\tilde q_n(\theta)) \to 0 ~\text{in-}P_{\theta_0} \text{ probability as $n \to \infty$}. 
	\end{align}
	Consequently, the sequence of  \blue{1-R\'enyi}-approximate posteriors must also achieve zero \scKL\ divergence from the true posterior in the large sample limit with high probability. 
	Finally, as demonstrated in Lemma~\ref{lem:subopt1ep}, any other sequence of  distribution that converges weakly to a distribution, that has positive probability at any point other that $\theta_0$ cannot achieve zero \scKL\ divergence. Therefore, it follows that the \blue{1-R\'enyi}-approximate posterior $s^*_n(\theta)$ must converge weakly to a Dirac delta distribution at the true parameter $\theta_0$, $\text{in-}P_{\theta_0} \text{ probability as $n \to \infty$}$, thereby completing the proof. 
\end{proof}}

\section{Models with Local Latent Parameters}\label{sec:latent}
We generalize the model we have worked with so far to include a collection of $n$ independent local latent variables $z_{1:n}:=\{z_1,z_2,\ldots,z_n\} \in \sZ^n$, one for each observation $\xi_i$. We assume these are distributed as {$\pi(z_i|\theta)$} for each  $i$, with the observations distributed as $p(\xi_i|z_i,\theta)$.  
Recall that $\theta$ is the global latent  variable with prior distribution $\pi(\theta)$. 
Denote by $z_0$ and $\theta_0$ the true local and global latent parameters respectively. For brevity we denote the model $P_{\theta_0,z_0}$ as $P_0$. The posterior distribution over $\theta$ and $\nZ$ is defined as
\[ \pi(\theta, \nZ| \mathbf X_n):= \frac{\pi(\theta)\prod_{i=1}^{n} {\pi(z_i|\theta)} p(\xi_i|z_i,\theta) }{ \int\int  \pi(\theta)\prod_{i=1}^{n}{\pi(z_i|\theta)} p(\xi_i|z_i,\theta) d\theta d\nZ }.  \] 
We denote the denominator above as $P(\nX)$, the model \textit{evidence}, and the numerator  as $p(\theta,\nX,\nZ)$. Since computing $P(\nX)$ is difficult, an approximate posterior can be 
obtained  by  minimizing the following objective over an appropriately chosen variational family $\cQ$:
\begin{align*}
 D_{\alpha} \left( \pi(\theta, \nZ| \mathbf X_n)\| {q}(\theta,  \nZ) \right):= \frac{1}{\alpha-1} 
\log \int_{\Theta \times \sZ^n} {q}(\theta, \nZ  )   \left(\frac{\pi(\theta, \nZ| \mathbf X_n)}{{q}(\theta, \nZ) } \right)^{\alpha} d\theta d\nZ,~\text{where $\alpha>1$}.
\end{align*}
This objective can be derived as an upper-bound to the model evidence similar to~\eqref{eq:renyi}.
{It is common to assume} that the variational family $\cQ$ factorizes into components $\cQ^n$ (over local variables) and $\bar \cQ$ (over $\theta$). 
Define the R\'enyi approximate posterior over the global parameter $\theta$ as
\begin{align}
\qnvr &:=  \argmin_{q(\theta)\in \bar \cQ} \min_{  q(\nZ)\in \cQ^n} \log \int_{\Theta \times \sZ^n} q(\theta)q(\nZ)   \left(\frac{p(\theta, \nZ, \mathbf X_n)}{q(\theta)q(\nZ) } \right)^{\alpha} d\theta d\nZ.
\label{eq:eqV0}
\end{align}
In this section, we aim to show that $\qnvr$ converges weakly to the Dirac delta distribution at $\theta_0$. To show this we require some additional assumptions.
%
%
First, define the profile likelihood at $\theta= \theta_0+ n^{-1/2}h_n$ for any bounded and stochastic $h_n= O_{P_0}(1)$ as $p(\nX|\theta_0+ n^{-1/2}h_n,\nZ^p)$, where $\nZ^p= \argmax_{\nZ}~p(\nX|\theta_0+ n^{-1/2}h_n,\nZ)$ is the maximum  profile likelihood estimate of $\nZ$ at $\theta= \theta_0+ n^{-1/2}h_n$. Denote $d_{H}(\nZ,\nZ^p):= H(P_{\theta_0,\nZ},P_{\theta_0,\nZ^p})$ as the Helinger distance between models $P_{\theta_0,\nZ}$ and $P_{\theta_0,\nZ^p}$ 
Furthermore, for any $\rho>0$ and for all bounded and stochastic $h_n= O_{P_0}(1)$, define $D(\theta_0+ n^{-1/2}h_n,\rho) = \{\nZ: d_{H}(\nZ,\nZ^p) < \rho  \}$ as the Hellinger ball of radius $\rho$ around $\nZ^p$.

Next we impose regularity conditions on the conditioned posterior $p(z_{1:n}|\nX,\theta_0)$. 
The assumption below follows~\citet[Proposition 10]{WaBl2017}, and is motivated by~\citet[Theorem 4.2]{bickel2012semiparametric}. 
\vspace{1em}
\begin{assumption}[Conditioned latent posterior]\label{assume:clp}
    The conditioned latent posterior $p(z_{1:n}|\nX,\theta_0)$  satisfies
    \begin{enumerate} 
        \item 
        
        The conditioned latent posterior is consistent under $n^{-1/2}$-perturbation at some rate $\rho_n$ with $\rho_n \downarrow 0 $ and $n\rho_n^2 \to \infty$, that is, for all bounded, stochastic $h_n= O_{P_{0}}(1) $, $p(z_{1:n}|\nX,\theta_0)$ converges as
        \[ \int_{D^c(\theta_0+ n^{-1/2}h_n,\rho_n)} p(z_{1:n}|\nX,\theta = \theta_0+ n^{-1/2}h_n ) d\nZ = o_{P_{0}}(1) .\]

        \item The sequence $\{\rho_n\}$ as defined above should also satisfy the following  conditions for all bounded and stochastic $h_n= O_{P_{0}}(1) $:
        \begin{align*}
        &(i)~\sup_{ \nZ\in  \{\nZ: d_{H}(\nZ,\nZ^p) < \rho_n  \} } \bbE_{P_{\theta_0,\nZ}}\left[\frac{p(\nX|\nZ,\theta_0+n^{-1/2}h_n )}{p(\nX|\nZ,\theta_0)}\right] = O(1),
        &(ii)~ d_{H}(z_0,\nZ^p)= o(\rho_n). 
        \end{align*}
    \end{enumerate}
\end{assumption}
 The first condition  ensures that conditioned latent posterior converges slower than the true posterior and the second condition is an additional regularity condition on the expected likelihood ratio.~\citet[Lemma 4.3]{bickel2012semiparametric} identifies mild differentiablity conditions on the likelihood ratio that imply condition 2(i) above. Also, Theorem  3.1 in~\citet{bickel2012semiparametric} provide the regularity conditions under which the the conditioned latent posterior satisfies the first condition above.

The next assumption, adapted from~\citet{bickel2012semiparametric}, is an extension of LAN condition in Assumption~\ref{assume:lan} to models with both global and local latent parameters. 
\vspace{1em}
\begin{assumption}[Stochastic LAN (s-LAN)]\label{assump:SLAN}
    Fix $\theta_0 \in \Theta$ and recall that $\nZ^p$ is the profile likelihood maximizer. The sequence of log-likelihood functions $\{ P^n_{\theta_0,\nZ^p}:= p(\nX|\theta_0,\nZ^p)\} $ satisfies \emph{stochastic local asymptotic normality (s-LAN)} condition if there exists a matrix $I(\theta_0,z_0)$ and a sequence of random vectors $\{\Delta_{n,(\theta_0,z_0)}\} \in L_2(P^n_{\theta_0,\nZ})$ such that for every bounded and stochastic sequence $\{h_n\}$, that is $h_n= O_{P_{0}}(1) $, we have 
    \[
    \log \frac{P^n_{\theta_0 + n^{-1/2} h_n,\nZ^p}}{P^n_{\theta_0 ,\nZ^p}} =  h_n^T I(\theta_0,z_0)
    \Delta_{n,(\theta_0,z_0)} - \frac{1}{2} h_n^T I(\theta_0,z_0)h_n  + o_{P_0}(1) ,
    \]
    where $P_0= P_{\theta_0,z_0}$. 
\end{assumption}
Stochastic LAN is slightly stronger than the usual LAN property. In most of the  examples, the ordinary LAN property often extends to stochastic LAN without significant difficulties~\citep{bickel2012semiparametric}. Also, Theorem 1 in~\citet{Murphy2000} identifies conditions under which the above LAN assumption  is satisfied by models with both global and local latent variables. It must be noted that if $\hat \theta_n$ is an asymptotically efficient estimator of $\theta_0$, then according to Lemma 25.25 in~\cite{vdV00} 
\( \sqrt{n} \left(\hat \theta_n -\theta_0\right) = 
\Delta_{n,(\theta_0,z_0)} + o_{P_0}(1). \)

Next we state a modified version of Assumption~\ref{def:gsequence}(3) for the models that contain local latent variables:
\vspace{1em}
\begin{assumption}[Good Sequence-Local]\label{def:gsequenceL}
    For any $\bar M >0$, the variational family $\bar \cQ$ contains a sequence of distributions $\{\bar{q}_n(\theta)\}$ with the following properties:
    \begin{enumerate}
        \item[(1)] there exists $n_1 \geq 1$ such that $\int_{\Theta}\theta \bar q_n(\theta) d\theta = \hat{\theta}_n $, where $\hat{\theta}_n$ is the maximum likelihood estimate, for each $n\geq n_1$, 
        \item[(2)] there exists $n_{\bar{M}}\geq 1$ such that the rate of convergence is $\gamma_n = \sqrt{n}$, that is \( E_{\bar q_n(\theta)}[|\theta - \hat \theta_n|^2] \leq  \frac{\bar M}{\gamma_n^2}  \) for each $n\geq n_{\bar M}$,
        \item[(3)] there exist a compact ball $K \subset  \Theta$ containing the true parameter $\theta_0$ and $n_2\geq1$, such that the sequence of Radon-Nikodym derivatives of the Bayes posterior density with respect to the sequence $\{\bar{q}_n\}$ exists and is bounded above by a finite positive constant $M_r$ outside of $K$ for all $n \geq n_2$ ; that is,\[\frac{  \ \pi(  \theta | \nX,\nZ^0)}{  \bar{q}_n(\theta)} \leq M_r, \ \forall \theta \in \Theta \backslash K \text{ and } \forall n\geq n_2, \quad P_{\theta_0}-a.s, \]
        where $\nZ^0$ is the first $n$ components of the true local latent parameter $z_0$. 
        \item[(4)] there exists $n_3\geq 1$ such that the good sequence $\{\bar q_n(\theta)\}$ is log-concave in $\theta$ for all $n \geq n_3$.
    \end{enumerate} 
\end{assumption}
\vspace{1em}
\begin{example}[Bayesian mixture model]
    Consider a mixture of uncorrelated $L$ univariate Gaussians, each with mean $\mu_i, i\in\{1,2,\ldots, L\}$ and unit variance. Each observation $X_i$ is assumed to be generated using the following model:
    \begin{align*}
    &\mu_l \sim \pi, \forall l\in \{1,2,
    \ldots,L\}
    \\
    & z_i \sim \text{Categorical}\left(\frac{1}{L},\frac{1}{L},\ldots,\frac{1}{L}\right), \forall i\in \{1,2,
    \ldots,n\}
    \\
    & X_i \sim \mathcal{N}(z_i^T \pmb \mu, 1) \forall i\in \{1,2,
    \ldots,n\}
    \end{align*} 
    Notice that $\pmb \mu$ is the global and $\nZ$ are the local latent parameters. Now observe that 
    \begin{align}
    \nonumber
    \pi(  \pmb \mu | \nX,\nZ^0) = \frac{\prod_{l=1}^{L}\pi(\mu_l) \prod_{i=1}^{n}p(z_i^0,X_i|\pmb \mu)}{\int \prod_{l=1}^{L}\pi(\mu_l) \prod_{i=1}^{n}p(z_i^0,X_i|\pmb \mu) d\pmb \mu} &= \frac{\prod_{l=1}^{L}\pi(\mu_l) \prod_{i=1}^{n}p(X_i|\pmb \mu,z_i^0)}{\int \prod_{l=1}^{L}\pi(\mu_l) \prod_{i=1}^{n}p(X_i|\pmb \mu,z_i^0) d\pmb \mu}\\
    &= \frac{\prod_{l=1}^{L}\pi(\mu_l) \prod_{i=1}^{n}\mathcal{N}(X_i|\pmb \mu^{T}z_i^0,1)}{\int \prod_{l=1}^{L}\pi(\mu_l) \prod_{i=1}^{n}\mathcal{N}(X_i|\pmb \mu^{T}z_i^0,1) d\pmb \mu}
    \\
    &= \frac{\prod_{l=1}^{L}\left[\pi(\mu_l) \prod_{j=1}^{n_l}\mathcal{N}(X_j^l|\mu_l,1)\right]}{\int \prod_{l=1}^{L}\pi(\mu_l) \prod_{j=1}^{n_l}\mathcal{N}(X_j^l|\mu_l,1) d\pmb \mu},
    \label{eq:eqa7}
    \end{align}
    where $X_j^l$ is the $j^{th}$ observation in the  $l^{th}$ cluster and $n_l= \sum_{i=1}^{n} z_{i,l}^0 $ is the total number of observations in the  $l^{th}$ cluster. 
  In practice, $\pi(\mu_l) =  \mathcal{N}(\mu_l|m,\sigma^2)\}$ is assumed to be a conjugate Gaussian with known mean $m$ and variance $\sigma^2$. 
  In this case, the distribution in~\eqref{eq:eqa7} can be computed analytically, that is 
    \begin{align*}
    \pi(  \mu | \nX,\nZ^0) = \frac{\prod_{l=1}^{L}\pi(\mu_l) \prod_{i=1}^{n}p(z_i^0,X_i|\mu)}{\int \prod_{l=1}^{L}\pi(\mu_l) \prod_{i=1}^{n}p(z_i^0,X_i|\mu) d\mu} &= \prod_{l=1}^{L} \mathcal{N}\left(\mu_l\bigg|\frac{1}{\frac{1}{\sigma^2}+{n_l}}\left( \frac{m}{\sigma^2}+\sum_{j=1}^{n_l}X_j^l\right) , \left(\frac{1}{\sigma^2}+n_l\right)^{-1}\right).
    \end{align*}
    In practice $\bar \cQ$ is chosen to be a mean-field approximate family, viz.\ a product of $L$ univariate Gaussians. 
    Now consider the following sequence of distributions in $\bar \cQ$
    \[q_n(\mu) = \prod_{l=1}^{L} \mathcal{N}\left(\mu_l|m_{n,l},\sigma^2_{n,l}\right) .\]  
    Choosing $m_{n,l}= \frac{1}{\frac{1}{\sigma^2}+{n_l}}\left( \frac{m}{\sigma^2}+\sum_{j=1}^{n_l}X_j^l\right) $ and $\sigma^2_{n,l}= \left(\frac{1}{\sigma^2}+n_l\right)^{-1}$, the ratio $\frac{\pi(  \mu | \nX,\nZ^0)}{  \bar{q}_n(\theta)} $ is bounded by $1$.
    
    The s-LAN assumption for finite mixtures model follows from the finiteness of the support of local latent variables~\citep{Murphy1996,Murphy2000}. 
\end{example}


In the next result we show that a consistent sequence asymptotically achieves zero $\alpha-$R\'enyi divergence. To show its existence, we first provide an asymptotic upper-bound on the minimum of the LHS in~\eqref{eq:eqa5} in the next proposition. This will allow us to prove the consistency of the minimizing sequence.
\vspace{1em}
\begin{proposition}\label{prop:UBfinL}
    For a given $\a>1$ and under Assumptions~\ref{assume:prior},~\ref{assume:var} (for $\bar \cQ$),~\ref{assume:clp},~\ref{assump:SLAN}, ~\ref{def:gsequenceL}, and for any~\good~there exist $n_0\geq 1$ and $\bar M > 0$ such that for all $n \geq n_0$, the minimal $\alpha-$R\'enyi divergence satisfies  
    \begin{align}
    \nonumber
    \min_{q \in \bar\cQ} \min_{q(z_{1:n})\in \cQ^n }  D_{\alpha}(\pi(\theta,z_{1:n}|\nX) \|q(\theta)q(z_{1:n}))&\leq \min_{q(z_{1:n})\in \cQ^n }  D_{\alpha}(\pi(\theta,z_{1:n}|\nX) \|\bar q_n(\theta)q(z_{1:n}))  \\
    &\leq B = \frac{1}{2}\log\left(\frac{\bar e \bar M I(\theta_0,z_0)}{\alpha^{\frac{1}{\alpha-1}}}\right) + o_{P_0}(1)
    \label{eq:eq21}
    \end{align}
    where $\bar e$ is the Euler's constant and $I(\theta_0,z_0)$ is as defined in Assumption~\ref{assump:SLAN}.
\end{proposition}

Since the term on the RHS above in~\eqref{eq:eq21} is non-negative for all $n\geq n_0$, implying that  \( \bar M \geq  \frac{\a^{\frac{1}{\a-1}}}{\bar e I(\theta_0,z_0)}  \)  for all $n \geq n_0$.\
Therefore, a specific good sequence can be chosen by fixing \( \tilde M =  \frac{\a^{\frac{1}{\a-1}}}{\bar e I(\theta_0,z_0)}  \), implying that
\(
\limsup_{n \to \infty} \min_{q(z_{1:n})\in \cQ^n }  D_{\alpha}(\pi(\theta,z_{1:n}|\nX) \|\tilde  q_n(\theta)q(z_{1:n})) = 0  ~ \forall n \geq n_0.
\) Now analogous to the parametric case we are only left to show that the global R\'enyi approximator necessarily converges to a Dirac delta distribution concentrated at the true global parameter $\theta_0$ to achieve zero R\'enyi divergence.


Now notice that for any $n\geq 1$, 
\begin{align}
\nonumber
\begin{split}
\min_{q(\nZ)\in \cQ^n} &\log \int_{\Theta }  {q}(\theta) \left(\frac{\pi(\theta)}{q(\theta) } \right)^{\alpha}  \int_{ \sZ^n} q(\nZ) \left(\frac{p( \nZ, \mathbf X_n|\theta)}{q(\nZ) } \right)^{\alpha}  d\nZ  d\theta 
\\
&\geq \log \int_{\Theta }  {q}(\theta) \left(\frac{\pi(\theta)}{q(\theta) } \right)^{\alpha} \min_{q(\nZ)\in \cQ^n}  \int_{ \sZ^n} q(\nZ) \left(\frac{p( \nZ, \mathbf X_n|\theta)}{q(\nZ) } \right)^{\alpha}  d\nZ  d\theta
\end{split}
\\
&= \log \int_{\Theta }  {q}(\theta) \left(\frac{\pi(\theta)M(\nX|\theta)}{q(\theta) } \right)^{\alpha}   d\theta,
\label{eq:eqa4}
\end{align}

where $M(\nX|\theta)$ is the variational likelihood define as 
\begin{align}
M(\nX|\theta):= \left[\min_{q(\nZ)\in \cQ^n} \int_{ \sZ^n} q(\nZ) \left(\frac{p( \nZ, \mathbf X_n|\theta)}{q(\nZ) } \right)^{\alpha}  d\nZ \right]^{1/\alpha}.
\label{eq:VL}
\end{align}
Observe that subtracting the $\log P(\nX)^\alpha$ from either side of~\eqref{eq:eqa4} yields:
\begin{align}
\min_{q(z_{1:n})\in \cQ^n }  D_{\alpha}(\pi(\theta,z_{1:n}|\nX) \|q(\theta)q(z_{1:n})) \geq D_{\alpha}(\pi^*(\theta|\nX) \|q(\theta)),
\label{eq:eqa5}
\end{align}
where the ideal posterior $\pi^*(\theta| \mathbf X_n)$ is defined as 
\begin{align}
\pi^*(\theta| \mathbf X_n)  := \frac{\pi(\theta)M(\nX|\theta)}{\int \pi(\theta)M(\nX|\theta) d\theta}.
\end{align}


%
In the subsequent lemma we show that  under certain regularity conditions  $M(\nX|\theta)$ satisfies the LAN condition with the similar expansion as of the true likelihood model for a given local latent parameter $z_0$. 
The proof parallels that of 
~\citet[Proposition 10]{WaBl2017}.
\vspace{1em}
\begin{lemma}\label{lem:LAN}
    Fix $\theta \in \Theta$. Under Assumptions~\ref{assume:clp} and ~\ref{assump:SLAN}, the sequence of variational log-likelihood functions $\{ M_n(\theta):=\log M(\nX|\theta)$ satisfies \emph{s-LAN} condition, that is
there exists a matrix $I(\theta_0,z_0)$ and a sequence of random vectors $\{\Delta_{n,(\theta_0,z_0)}\}$ as defined in Assumption~\ref{assump:SLAN}, such that for every bounded and stochastic sequence $\{h_n\}$, that is $h_n= O_{P_{0}}(1) $, we have 
    \[
    \log \frac{M_n(\theta_0 + n^{-1/2} h_n)}{M_n(\theta_0)} =  h_n^T I(\theta_0,z_0)
    \Delta_{n,(\theta_0,z_0)} - \frac{1}{2} h_n^T I(\theta_0,z_0)h_n  + o_{P_0}(1).
    \]
\end{lemma}

Next, we will show that the minimizing sequence must converge to a Dirac delta distribution at $\theta_0$ using the  results in Proposition~\ref{prop:UBfinL} and Lemma~\ref{lem:LAN}.
\vspace{1em} 
\begin{theorem}\label{thm:consistencyL}
    For a given $\a>1$ and under Assumptions~\ref{assume:prior},~\ref{assume:var} (for $\bar \cQ$) ,~\ref{assume:clp}, and~\ref{def:gsequenceL}, the $\alpha-$R\'enyi approximate posterior $q^*_n(\theta)$ over global latent parameters $\theta$ as defined in~\eqref{eq:eqV0} converges weakly to a Dirac delta  distribution at the true parameter $\theta_0$; that is, 
    \[q^*_n(\theta)  \Rightarrow \delta_{\theta_0}\ \text{ in}~P_{0}-\text{probability as $n \to \infty$.} \]
\end{theorem}
\begin{proof}
Using the result in Proposition~\ref{prop:UBfinL} and following similar steps as used in Theorem~\ref{thm:consistency}, we can show that the minimizing sequence must have zero $\alpha$-R\'enyi divergence in the limit with high probability. 
Recall the inequality in~\eqref{eq:eqa5}
\begin{align}
\min_{q(z_{1:n})\in \cQ^n }  D_{\alpha}(\pi(\theta,z_{1:n}|\nX) \|q(\theta)q(z_{1:n})) \geq D_{\alpha}(\pi^*(\theta|\nX) \|q(\theta)).
\label{eq:eqa5a}
\end{align}
Also note that $q^*_n(\theta)$ is the minimizer of the LHS in the equation above.
Since the variational likelihood satisfies the LAN condition due to Lemma~\ref{lem:LAN}, under the consistent testability assumption, the ideal posterior $\pi^*(\theta|\nX) $ also degenerates to a Dirac delta distribution at the true parameter $\theta_0$~\citep{Kleijn2012}. 

Now recall Lemma~\ref{lem:ndegen} and~\ref{lem:subopt1}. Following the arguments in Lemma~\ref{lem:ndegen}, and  using the inequality in~\eqref{eq:eqa5a} we can argue that any sequence of distributions in $\bar \cQ$ that minimizes the LHS in~\eqref{eq:eqa5a} must converge weakly to a  Dirac delta distribution at the true parameter $\theta_0$ in the large sample limit, since otherwise the objective in the LHS of~\eqref{eq:eqa5a} is unbounded. In addition, using Lemma~\ref{lem:subopt1} and the inequality in~\eqref{eq:eqa5a} we can also show that any sequence of distribution in $\bar \cQ$ that converges weakly to a convex combination of a Dirac delta distribution at $\theta_0$ with any other distribution can not achieve zero $\alpha-$R\'enyi divergence in the limit. This completes the proof.
\end{proof}


\section*{Acknowledgments}
This research is supported by the National Science Foundation (NSF) through awards DMS/1812197 and IIS/1816499, and the Purdue Research Foundation (PRF). 


\appendix
\section{Proofs}

\subsection{Proofs in Section~\ref{sec:asymptote}}
We begin with the following well known result.
\vspace{0.5em}
\begin{lemma}\label{lem:laplace}[Laplace Approximation] 
	Consider an integral of the form 
	\[I = \int_{a}^{b} h(y)e^{-ng(y)} dy,\]
	where $g(y)$ is a smooth function which has a local minimum at $y^* \in (a,b)$ and $h(y)$ is a smooth function. Then
	\[I \sim h(y^*) e^{-n g(y^*)} \sqrt{\frac{2\pi}{n g''(y^*)}} \text{ as $n \to \infty$} . \]
\end{lemma}

\begin{proof}
	Readers are directed to \citet[Chapter-2]{WONG198955} for the proof.
\end{proof}

Now we prove a technical lemma that bounds the differential entropy of the good sequence.
	\begin{lemma}\label{lem:entropy}
		For a good sequence $\bar q_n(\theta)$, there exist an $n_M\geq 1$ and $\bar M >0 $, such that for all $n \geq n_M$
		\[ - \int \bar{q}_n(\mu) \log \bar{q}_n(\mu)  \leq  \frac{1}{2}\log \left(2\pi \bar e \frac{\bar M}{n} \right),  \]	
	where $\bar e$ is the Euler's constant.
	\end{lemma}

\begin{proof}
	Recall from Assumption~\ref{def:gsequence} that the $\bar q_n(\theta)$ converges weakly to $\delta_{\theta_0}$ at the rate of  $\sqrt{n}$. It follows from the Definition~\ref{def:roc} for rate of convergence that,
	\[E_{\bar q_n(\theta)}[|\theta - \hat \theta_n|^2] = O\left (\frac{1}{n} \right). \]  

	There exist an $n_M\geq 1$ and $\bar M > 0 $, such that for all $n \geq n_M$
	\[  \bbE_{\bar q_n(\theta)} [( \theta - \hat{\theta}_n)^2 ] \leq  \frac{\bar M}{n}. \] 
	Using the fact that, the differential entropy  of random variable with a given variance is bounded by the differential entropy of the Gausian distribution of the same variance \cite[Theorem 9.6.5]{CoverTM2006}),
	it follows that the differential entropy of $\bar{q}_{n}(\mu)$ is bounded by $\frac{1}{2}\log(2\pi \bar e \frac{\bar M}{n}) $, where $\bar e$ is the Euler's constant.
\end{proof}

    Next, we prove the following result on the prior distributions. This result will be useful in proving Lemma~\ref{lem:norm} and~\ref{lem:ndegen}.
\vspace{0.5em}
\begin{lemma}\label{lem:prior}
	Given a prior distribution  $\pi(\theta)$ with $\bbE_{ \pi(\theta)}[|\theta|]< \infty$, for any $\beta>0$, there exists a sequence of compact sets $\{K_n\} \subset \Theta$ such that 
	\[\int_{\Theta  \backslash K_n }  \pi(\gamma) d\gamma = O(n^{-\beta}). \]
\end{lemma}

\begin{proof}
	Fix $\theta_1 \in \Theta$. Define a sequence of compact sets $$K_n = \{\theta \in \Theta : |\theta - 
	\theta_1| \leq n^{{\beta}} \} \forall \beta>0.$$ Clearly, as $n$ increases $K_n$ approaches $\Theta$. Now, using Markov's inequality followed by the triangule inequality,  
	\begin{align}
	\nonumber
	\int_{\Theta  \backslash K_n }  \pi(\gamma) d\gamma = \int_{\{\gamma \in \Theta : |\gamma - 
		\theta_1| > n^{{\beta}} \} } \pi(\gamma) d\gamma &\leq n^{-\beta} \bbE_{ \pi(\theta)}[|\gamma - 
	\theta_1|]
	\\ & \leq n^{-\beta} \left( \bbE_{ \pi(\theta)}[|\gamma|] + |\theta_1| \right).
	\end{align}  
	Since, $ \bbE_{ \pi(\gamma)}[|\gamma|] < \infty$, it follows that $\forall \beta>0 $,
	\(\int_{\Theta  \backslash K_n }  \pi(\gamma) d\gamma = O(n^{-\beta}) .\)
\end{proof}

The next result approximates the normalizing sequence of the posterior distribution using the lemma above and the LAN condition. 
\vspace{1em}
\begin{lemma}\label{lem:norm}
	There exists a sequence of compact balls $\{K_n \subset \Theta\}$, such that $\theta_0 \in K_n$ and under Assumptions~\ref{assume:prior} and~\ref{assume:lan}, the normalizing sequence of the posterior distribution
	\begin{align}
	\nonumber
	\int_{\Theta} &  \prod_{i=1}^n \frac{p(X_i|\gamma)}{p(X_i|\theta_0)}   \pi(\gamma) d\gamma 
	\\
	=&   \sqrt{\frac{2\pi}{nI(\theta_0)}} e^{\left( \frac{1}{2} n I(\theta_0) \left( (\hat{\theta}_n - \theta_0 )^2 \right) \right)}    \bigg( e^{o_{P_{\theta_0}}(1)} \int_{K_n} \pi(\gamma) \mathcal{N}(\gamma;\hat{\theta}_n,(n I(\theta_0))^{-1}) d\gamma + o(1) \bigg).
	\label{eq:eqp6a}
	\end{align}
	\end{lemma}
\begin{proof}
	Let $\{K_n \subset \Theta\}$ be a sequence of compact balls such that $\theta_0 \in K_n$, where $\theta_0$ is any point in $\Theta$ where prior distribution $\pi(\theta)$ places positive density. Using Lemma~\ref{lem:prior}, we can always find a sequence of sets $\{K_n\}$ for a prior distribution, such that $\theta_0 \in K_n$ and for any positive constant $\beta>\frac{3}{2}$,
	\begin{align}
	\int_{\Theta  \backslash K_n }  \pi(\gamma) d\gamma = O(n^{-\beta}) .\label{eq:eqp1}
	\end{align}
	Observe that
	\begin{align}
	\int_{\Theta}   \prod_{i=1}^n \frac{p(X_i|\gamma)}{p(X_i|\theta_0)}   \pi(\gamma) d\gamma &= \left( \int_{K_n}   \prod_{i=1}^n \frac{p(X_i|\gamma)}{p(X_i|\theta_0)}   \pi(\gamma) d\gamma + \int_{\Theta  \backslash K_n }  \pi(\gamma) \prod_{i=1}^n \frac{p(X_i|\gamma)}{p(X_i|\theta_0)}    d\gamma \right).
	\label{eq:eqp2}
	\end{align}
	
	Consider the first term in~\eqref{eq:eqp2}; following similar steps as in~\eqref{eq:eq11b} and~\eqref{eq:eq13} and using Assumption~\ref{assume:lan}, we have
	\begin{align}
	\nonumber
	\int_{K_n} & \prod_{i=1}^n \frac{p(X_i|\gamma)}{p(X_i|\theta_0)}   \pi(\gamma) d\gamma 
	\\
	\nonumber	
	&=   e^{o_{P_{\theta_0}}(1)} \exp \left( \frac{1}{2} n I(\theta_0) \left((\hat{\theta}_n - \theta_0 )^2 \right) \right) \int_{K_n} \pi(\gamma) \exp \left( -\frac{1}{2} n I(\theta_0) \left((\gamma - \hat{\theta}_n)^2 \right) \right) d\gamma
	\\
	& = e^{o_{P_{\theta_0}}(1)} \exp \left( \frac{1}{2} n I(\theta_0) \left((\hat{\theta}_n - \theta_0 )^2 \right) \right)  \sqrt{\frac{2\pi}{nI(\theta_0)}} \int_{K_n} \pi(\gamma) \mathcal{N}(\gamma;\hat{\theta}_n,(n I(\theta_0))^{-1}) d\gamma,
	\label{eq:eqp3}
	\end{align}
	where the last equality follows from the definition of Gaussian density, $\mathcal{N}(\cdot;\hat{\theta}_n,(n I(\theta_0))^{-1})$.
	
	Substituting~\eqref{eq:eqp3} into~\eqref{eq:eqp2}, we obtain
	\begin{align}
	\nonumber
	\int_{\Theta} &  \prod_{i=1}^n \frac{p(X_i|\gamma)}{p(X_i|\theta_0)}   \pi(\gamma) d\gamma 
	\\
	\nonumber
	=&   \exp \left( \frac{1}{2} n I(\theta_0) \left( (\hat{\theta}_n - \theta_0 )^2 \right) \right)  \sqrt{\frac{2\pi}{nI(\theta_0)}}  \bigg( e^{o_{P_{\theta_0}}(1)} \int_{K_n} \pi(\gamma) \mathcal{N}(\gamma;\hat{\theta}_n,(n I(\theta_0))^{-1}) d\gamma 
	\\
	&+ \exp \left( -\frac{1}{2} n I(\theta_0) \left( (\hat{\theta}_n - \theta_0 )^2 \right) \right)  \sqrt{\frac{nI(\theta_0)}{2\pi}}  \int_{\Theta  \backslash K_n }  \pi(\gamma) \prod_{i=1}^n \frac{p(X_i|\gamma)}{p(X_i|\theta_0)}    d\gamma \bigg).
	\label{eq:eqp4}
	\end{align}
	
	Next, using the Markov's inequality and then Fubini's Theorem, for arbitrary $\delta>0$, we have
	\begin{align}
	\nonumber
	P_{\theta_0} \left( \sqrt{ \frac{n I(\theta_0)}{2 \pi } } \int_{\Theta  \backslash K_n }   \prod_{i=1}^n \frac{p(X_i|\gamma)}{p(X_i|\theta_0)}   \pi(\gamma) d\gamma  > \delta \right) &\leq \sqrt{ \frac{  n I(\theta_0)}{ \delta^2 2 \pi } } \bbE_{P_{\theta_0}} \left[ \int_{\Theta  \backslash K_n }   \prod_{i=1}^n \frac{p(X_i|\gamma)}{p(X_i|\theta_0)}   \pi(\gamma) d\gamma  \right]
	\\
	\nonumber
	&= \sqrt{ \frac{  n I(\theta_0)}{ \delta^2 2 \pi } }  \int_{\Theta  \backslash K_n } \bbE_{P_{\theta_0}} \left[    \prod_{i=1}^n \frac{p(X_i|\gamma)}{p(X_i|\theta_0)} \right]  \pi(\gamma) d\gamma  
	\\
	&= \sqrt{ \frac{  n I(\theta_0)}{ \delta^2 2 \pi } }  \int_{\Theta  \backslash K_n }  \pi(\gamma) d\gamma,
	\end{align}
	since \(\bbE_{P_{\theta_0}} \left[    \prod_{i=1}^n \frac{p(X_i|\gamma)}{p(X_i|\theta_0)} \right]  = 1.\) 
	
	Hence, using~\eqref{eq:eqp1} for $\beta>3/2$, it is straightforward to observe that
	\[  P_{\theta_0} \left( \sqrt{ \frac{n I(\theta_0)}{2 \pi } } \int_{\Theta  \backslash K_n }   \prod_{i=1}^n \frac{p(X_i|\gamma)}{p(X_i|\theta_0)}   \pi(\gamma) d\gamma  > \delta \right) \leq  \sqrt{ \frac{ I(\theta_0)}{ \delta^2 2 \pi } } \frac{1}{n^{\b- 1/2}}.   \]
	
	Since the upper bound above is summable, using First Borel-Cantelli Theorem it follows that   
	\begin{align} \sqrt{ \frac{n I(\theta_0)}{2 \pi } } \int_{\Theta  \backslash K_n }   \prod_{i=1}^n \frac{p(X_i|\gamma)}{p(X_i|\theta_0)}   \pi(\gamma) d\gamma = o(1) ~P_{\theta_0}-\text{a.s.}. 
	\label{eq:eqp5}
	\end{align}
	
	Since, $\exp \left( -\frac{1}{2} n I(\theta_0) \left( (\hat{\theta}_n - \theta_0 )^2 \right) \right)  \leq 1 $, it follows from substituting~\eqref{eq:eqp5} into~\eqref{eq:eqp4} that
	\begin{align*}
	\int_{\Theta} &  \prod_{i=1}^n \frac{p(X_i|\gamma)}{p(X_i|\theta_0)}   \pi(\gamma) d\gamma 
	\\
	=&   \exp \left( \frac{1}{2} n I(\theta_0) \left( (\hat{\theta}_n - \theta_0 )^2 \right) \right)  \sqrt{\frac{2\pi}{nI(\theta_0)}}  \bigg( e^{o_{P_{\theta_0}}(1)} \int_{K_n} \pi(\gamma) \mathcal{N}(\gamma;\hat{\theta}_n,(n I(\theta_0))^{-1}) d\gamma + o(1) \bigg).
	\end{align*}
	\end{proof}
Next we prove Lemma~\ref{lem:ndegen}, showing that the $\alpha-$R\'enyi divergence between the posterior and any non-degenerate distribution diverges in the large sample limit.
\begin{proof}[\textbf{Proof of Lemma~\ref{lem:ndegen}}]
	Let $K_n \subset \Theta$ be a sequence of compact sets such that $\theta_0 \in K_n$, where $\theta_0$ is any point in $\Theta$ where prior distribution $\pi(\theta)$ places positive density. Using Lemma~\ref{lem:prior}, we can always find a sequence of sets $\{K_n\}$ for a prior distribution, such that $\theta_0 \in K_n$ and for any positive constant $\beta>\frac{1}{2}$,
	\begin{align}
	\int_{\Theta  \backslash K_n }  \pi(\gamma) d\gamma = O(n^{-\beta}) .\label{eq:eq1a}
	\end{align}
	
	Now, observe that
	\begin{align}
	\nonumber
	\frac{\alpha-1}{\a}&D_\a(\pi(\theta|\nX)\| q_n(\theta)) 
	\\ 
	\nonumber
	&= \frac{1}{\alpha} \log \left  ( \int_{K_n} q_n(\theta) \left ( \frac{\pi(\theta|\nX)}{q_n(\theta)} \right)^{\alpha} d \theta + \int_{\Theta \backslash K_n} q_n(\theta) \left ( \frac{\pi(\theta|\nX)}{q_n(\theta)} \right)^{\alpha} d \theta \right)
	\\
	&\geq \frac{1}{\alpha} \log \left  ( \int_{K_n} q_n(\theta) \left ( \frac{\pi(\theta|\nX)}{q_n(\theta)} \right)^{\alpha} d \theta \right),
	\label{eq:eq11}
	\end{align}
	
	where the last inequality follows from the fact that the integrand is always positive.
	
	Next, we approximate the ratio in the integrand on the right hand side of the above equation using the LAN condition in Assumption~\ref{assume:lan}. Let
	$\Delta_{n,\theta_0} := \sqrt{n}(\hat{\theta}_n - \theta_0)$, such that $\hat{\theta}_n  \to \theta_0$, $P_{\theta_0} -a.s.$ and $\Delta_{n,\theta_0}$ converges in distribution to $\mathcal{N}(0,I(\theta_0)^{-1})$. Re-parameterizing the expression with $\theta =
	\theta_0 + n^{-1/2} h$, we have
	\begin{align}
	\nonumber
    \int_{K_n} & q_n(\theta)  \left ( \frac{\pi(\theta|\nX)}{q_n(\theta)} \right)^{\alpha} d \theta = n^{-1/2} \int_{K_n}  q_n(\theta_0 + n^{-1/2}
	h)   \left(   \frac {\pi(\theta_0 + n^{-1/2}
		h) \prod_{i=1}^n \frac{p(X_i|(\theta_0  + n^{-1/2} h) )}{p(X_i|\theta_0)}}{ q_n(\theta_0 + n^{-1/2}
		h) \int_{\Theta}   \prod_{i=1}^n \frac{p(X_i|\gamma)}{p(X_i|\theta_0)}   \pi(\gamma) d\gamma} \right)^{\a}  dh \\
	&= n^{-1/2} \int_{K_n}  q_n(\theta_0 + n^{-1/2}
	h)   \left(   \frac {\pi(\theta_0 + n^{-1/2}
		h) \prod_{i=1}^n \frac{p(X_i|(\theta_0  + n^{-1/2} h) )}{p(X_i|\theta_0)}}{ q_n(\theta_0 + n^{-1/2}
		h) \int_{\Theta}   \prod_{i=1}^n \frac{p(X_i|\gamma)}{p(X_i|\theta_0)}   \pi(\gamma) d\gamma} \right)^{\a}  dh \\
	&= n^{-1/2} \int_{K_n}  q_n(\theta_0 + n^{-1/2} h)
	\bigg ( \pi(\theta_0 + n^{-1/2} h) \frac{ \exp (h
		I(\theta_0)\Delta_{n,\theta_0} -	\frac{1}{2}h^2I(\theta_0)+
		o_{P_{\theta_0}}(1) )}{ q_n(\theta_0 + n^{-1/2}
		h) \int_{\Theta}   \prod_{i=1}^n \frac{p(X_i|\gamma)}{p(X_i|\theta_0)}   \pi(\gamma) d\gamma} \bigg)^{\a} dh.
	\end{align}
	{Resubstituting $h = \sqrt{ n}(\theta - \theta_0)$ in the expression above and reverting to the previous parametrization,}
	\begin{align}
	\nonumber
	&= \int_{K_n}   q_n(\theta) \left ( \pi(\theta) \frac{ \exp \left(
		\sqrt{n}(\theta - \theta_0) I(\theta_0)\Delta_{n,\theta_0} -
		\frac{1}{2} n (\theta - \theta_0)^2 I(\theta_0) + o_{P_{\theta_0}}(1) \right) }{  q_n(\theta)\int_{\Theta}   \prod_{i=1}^n \frac{p(X_i|\gamma)}{p(X_i|\theta_0)}   \pi(\gamma) d\gamma} \right)^{\alpha} d\theta 
	\\
	\nonumber
	&= \int_{K_n}   q_n(\theta)  \left ( \pi(\theta) \frac{ e^{o_{P_{\theta_0}}(1)}  \exp \left( -\frac{1}{2} n I(\theta_0) \left((\theta - \theta_0)^2 - 2(\theta - \theta_0)(\hat{\theta}_n- \theta_0)\right) \right) }{  q_n(\theta) \int_{\Theta}   \prod_{i=1}^n \frac{p(X_i|\gamma)}{p(X_i|\theta_0)}   \pi(\gamma) d\gamma } \right)^{\a} d\theta.
	\\
	\intertext{Now completing the square by dividing and multiplying the numerator by $\exp \left( \frac{1}{2} n I(\theta_0) \left((\hat{\theta}_n - \theta_0 )^2 \right) \right)$ we obtain}
	\nonumber
	&= \int_{K_n}   q_n(\theta)  \left ( \pi(\theta) \frac{ e^{o_{P_{\theta_0}}(1)} \exp \left( \frac{1}{2} n I(\theta_0) \left((\hat{\theta}_n - \theta_0 )^2 \right) \right) \exp \left( -\frac{1}{2} n I(\theta_0) \left((\theta - \hat{\theta}_n)^2 \right) \right) }{  q_n(\theta) \int_{\Theta}   \prod_{i=1}^n \frac{p(X_i|\gamma)}{p(X_i|\theta_0)}   \pi(\gamma) d\gamma } \right)^{\a} d\theta
	\\
	&= \int_{K_n}   q_n(\theta)  \left ( \pi(\theta) \frac{ e^{o_{P_{\theta_0}}(1)} \exp \left( \frac{1}{2} n I(\theta_0) \left((\hat{\theta}_n - \theta_0 )^2 \right) \right) \sqrt{\frac{2\pi}{nI(\theta_0)}}  \mathcal{N}(\theta;\hat{\theta}_n,(n I(\theta_0))^{-1}) }{  q_n(\theta) \int_{\Theta}   \prod_{i=1}^n \frac{p(X_i|\gamma)}{p(X_i|\theta_0)}   \pi(\gamma) d\gamma } \right)^{\a} d\theta,
	\label{eq:eq12}
	\end{align}
	where, in the last equality we used the definition of Gaussian density, $\mathcal{N}(\cdot;\hat{\theta}_n,(n I(\theta_0))^{-1})$. 
	
	Next, we approximate the integral in the denominator of~\eqref{eq:eq13}. 
    Using Lemma~\ref{lem:norm}, it follows that there exist a sequence of compact balls $\{K_n \subset \Theta\}$, such that $\theta_0 \in K_n$ and
	\begin{align}
	\nonumber
	\int_{\Theta} &  \prod_{i=1}^n \frac{p(X_i|\gamma)}{p(X_i|\theta_0)}   \pi(\gamma) d\gamma 
	\\
	=&   \sqrt{\frac{2\pi}{nI(\theta_0)}} e^{\left( \frac{1}{2} n I(\theta_0) \left( (\hat{\theta}_n - \theta_0 )^2 \right) \right)}    \bigg( e^{o_{P_{\theta_0}}(1)} \int_{K_n} \pi(\gamma) \mathcal{N}(\gamma;\hat{\theta}_n,(n I
    (\theta_0))^{-1}) d\gamma + o(1) \bigg).
	\label{eq:eqp6b}
	\end{align}
	
	Substituting~\eqref{eq:eqp6b} into~\eqref{eq:eq12} and simplifying, we obtain
	\begin{align}
	\nonumber
		\int_{K_n}&  q_n(\theta)  \left ( \frac{\pi(\theta|\nX)}{q_n(\theta)} \right)^{\alpha} d \theta 
		\\
		&= \int_{K_n}   q_n(\theta)^{1-\a}  \left (  \frac{  e^{o_{P_{\theta_0}}(1)} \pi(\theta) \mathcal{N}(\theta;\hat{\theta}_n,(n I(\theta_0))^{-1}) }{  \bigg( e^{o_{P_{\theta_0}}(1)} \int_{K_n} \pi(\gamma) \mathcal{N}(\gamma;\hat{\theta}_n,(n I(\theta_0))^{-1}) d\gamma + o(1) \bigg) } \right)^{\a} d\theta.
		\label{eq:eq6c}
		\end{align}
	Observe that  
	\[ \left( \mathcal{N}(\theta;\hat{\theta}_n,(n I(\theta_0))^{-1}) \right)^{\a} = \left( \sqrt{ \frac{n I(\theta_0)}{2 \pi } } \right)^{\alpha} \left( \sqrt{\frac{2\pi}{n \alpha I(\theta_0)}}  \right)  \mathcal{N}(\theta;\hat{\theta}_n,(n \alpha I(\theta_0))^{-1}). \] 	
	Substituting this into the right hand side of \eqref{eq:eq6c}
	\begin{align}
	\nonumber
	& \frac{1}{\alpha} \log \int_{K_n} q_n(\theta)^{1-\a}  \left (  \frac{ \pi(\theta)  \mathcal{N}(\theta;\hat{\theta}_n,(n I(\theta_0))^{-1})  }{ \left( e^{o_{P_{\theta_0}}(1)}  \int_{K_n}  \pi(\gamma)  \mathcal{N}(\gamma;\hat{\theta}_n,(n I(\theta_0))^{-1})  d\gamma + o(1) \right)  }  \right)^{\a} d\theta
	\\
	\nonumber
	= & - \log  \left( e^{o_{P_{\theta_0}}(1)}  \int_{K_n}  \pi(\gamma)  \mathcal{N}(\gamma;\hat{\theta}_n,(n I(\theta_0))^{-1})  d\gamma + o(1) \right)  + \frac{\a-1}{2 \a} \log n - \frac{\log \a}{2\a}
	\\
	&   +\frac{\a-1}{2 \a} \log \frac{I(\theta_0)}{2 \pi } 
	+ \frac{1}{\alpha} \log \int_{K_n} q_n(\theta)^{1-\a} \pi(\theta)^{\a} \mathcal{N}(\theta;\hat{\theta}_n,(n \alpha I(\theta_0))^{-1})  d\theta. 
	\label{eq:eq31}
	\end{align}
	
	From the Laplace approximation (Lemma~\ref{lem:laplace}) and the continuity of the logarithm, we have 
	\[- \log  \left( e^{o_{P_{\theta_0}}(1)}  \int_{K_n}  \pi(\gamma)  \mathcal{N}(\gamma;\hat{\theta}_n,(n I(\theta_0))^{-1})  d\gamma + o(1) \right) \sim - \log  \left( e^{o_{P_{\theta_0}}(1)}  \pi(\hat{\theta}_n)   \right).\] Next, using the Laplace approximation on the last term in \eqref{eq:eq31}
	\begin{align}
	\nonumber
	\frac{1}{\alpha} \log \int_{K_n} q_n(\theta)^{1-\a} \pi(\theta)^{\a} \mathcal{N}(\theta;\hat{\theta}_n,(n \alpha I(\theta_0))^{-1})  d\theta
	\sim  \frac{\alpha -1}{\alpha} \log \frac{1}{q_n(\hat{\theta}_n)} +  \log \pi(\hat{\theta}_n) .
	\end{align}
	
	Substituting the above two approximations into~\eqref{eq:eq31}, we have
	\begin{align}
	\nonumber
	\frac{1}{\alpha} &\log \int_{K_n} q_n(\theta)^{1-\a}  \left (  \frac{ \pi(\theta)  \mathcal{N}(\theta;\hat{\theta}_n,(n I(\theta_0))^{-1})  }{ \left( e^{o_{P_{\theta_0}}(1)}  \int_{K_n}  \pi(\gamma)  \mathcal{N}(\gamma;\hat{\theta}_n,(n I(\theta_0))^{-1})  d\gamma + o(1) \right)  }  \right)^{\a} d\theta
	\\
	\nonumber
	\sim &  - \log  \left( e^{o_{P_{\theta_0}}(1)}  \pi(\hat{\theta}_n) \right)  - \frac{\log \a}{2\a} +\frac{\a-1}{2 \a} \log \frac{I(\theta_0)}{2 \pi }
	\\
	\nonumber
	&     +
	\frac{\alpha -1}{2\alpha} \log {n} - 
	\frac{\alpha -1}{\alpha} \log q_n(\hat{\theta}_n) + \log \pi(\hat{\theta}_n)
	\\
	\nonumber
	\sim &  - \log  \left(   \pi(\hat{\theta}_n)  \right)  - \frac{\log \a}{2\a} +\frac{\a-1}{2 \a} \log \frac{I(\theta_0)}{2 \pi } +
	\frac{\alpha -1}{2\alpha} \log {n} - 
	\frac{\alpha -1}{\alpha} \log q(\hat{\theta}_n) + \log \pi(\hat{\theta}_n) + o_{P_{\theta_0}}(1)
	\\
	= &    - \frac{\log \a}{2\a} +\frac{\a-1}{2 \a} \log \frac{I(\theta_0)}{2 \pi } +
	\frac{\alpha -1}{2\alpha} \log {n} - 
	\frac{\alpha -1}{\alpha} \log q(\hat{\theta}_n) + o_{P_{\theta_0}}(1),
	\label{eq:eq29a}
	\end{align}
	
    where the penultimate approximation follows from the fact that \[ 
    q_n(\hat{\theta}_n) \sim q(\hat{\theta}_n). \]
    Note that $\hat{\theta}_n \to \theta_0 ,~ P_{\theta_0}- a.s $. Therefore, if $q(\theta_0) = 0$, then the right hand side in~\eqref{eq:eq29a} will diverge as $n \to \infty$ because $\frac{\a-1}{2 \a} \log n $ also diverges as $n \to \infty$. Also observe that, for any $q(\theta)$ that places finite mass on $\theta_0$, the $\alpha-$R\'enyi divergence diverges as $n \to \infty$. Hence, $\alpha-$R\'enyi approximate posterior must converge weakly to a distribution that has a Dirac delta distribution at the true parameter $\theta_0$. 
\end{proof}

Next, we show that the $\alpha-$R\'enyi divergence between the true posterior and the sequence $\{q'_n(\theta)\}\in \cQ$ as defined in~\eqref{eq:eq14i} is bounded below by a positive number. 
\begin{proof}[\textbf{Proof of Lemma~\ref{lem:subopt1}}]
	\citet[Theorem 19]{TErvan2012} shows that for any $\alpha>0$, the $\alpha-$R\'enyi divergence  $D_\alpha(p(\theta)\| q(\theta))$ is a lower semi-continuous function of the pair $(p(\theta),q(\theta))$ in the weak topology on  the space of probability measures. Recall from~\eqref{eq:4} that the true posterior distribution $\pi(\theta|\nX)$ converges weakly to $\delta_{\theta_0}$~$P_{\theta_0}-a.s$. Using this fact it follows that
	\begin{align*} \liminf_{n \to \infty} D_\alpha(\pi(\theta|\nX) \| q'_n(\theta)) \geq D_\alpha \left(\delta_{\theta_0} \bigg\| w^j\delta_{\theta_0} +  \sum_{i=1, i\neq j}^{\infty} w^i q_i(\theta) \right)\quad P_{\theta_0}-a.s.
	\end{align*}
	Next, using Pinsker's inequality~\citep{CoverTM2006} for  $\alpha > 1$, we have
	\begin{align*} 
	D_\alpha\left(\delta_{\theta_0} \bigg\| w^j\delta_{\theta_0} +  \sum_{i=1, i\neq j}^{\infty} w^i q_i(\theta) \right) &\geq \frac{1}{2} \left( \int_{\Theta} \left | \delta_{\theta_0} - w^j\delta_{\theta_0} -  \sum_{i=1, i\neq j}^{\infty} w^i q_i(\theta) \right| d\theta \right)^2
	\\
	&=  \frac{1}{2} \left( \int_{\Theta} \left |(1-w^j) \delta_{\theta_0} - \sum_{i=1 , i \neq j}^{\infty} w^i q_i(\theta) \right | d\theta \right)^2.
	\end{align*}
	Now dividing the integral over ball of radius $\epsilon$ centered at $\theta_0$, $B(\theta_0,\epsilon)$ and its complement, we obtain  
	\begin{align}
	\nonumber
	\liminf_{n \to \infty}& D_\alpha(\pi(\theta|\nX)\| q'_n(\theta))
	\\
	\nonumber
	&\geq \frac{1}{2} \left( \int_{B(\theta_0,\epsilon)} \left |(1-w^j) \delta_{\theta_0} - \sum_{i=1 , i \neq j}^{\infty} w^i q_i(\theta) \right| d\theta + \int_{B(\theta_0,\epsilon)^C} \left|(1-w^j) \delta_{\theta_0} - \sum_{i=1 , i \neq j}^{\infty} w^i q_i(\theta) \right | d\theta \right)^2
	\\
	\nonumber
	&\geq  \frac{1}{2} \left( \int_{B(\theta_0,\epsilon)^C} \left |(1-w^j) \delta_{\theta_0} - \sum_{i=1 , i \neq j}^{\infty} w^i q_i(\theta) \right | d\theta \right)^2
	\\
	&=  \frac{1}{2} \left( \int_{B(\theta_0,\epsilon)^C} \left | - \sum_{i=1 , i \neq j}^{\infty} w^i q_i(\theta) \right | d\theta \right)^2 \quad P_{\theta_0}-a.s.
	\end{align}
	
	Since,  $w^i \in (0,1)$, observe that for any $\epsilon>0$, there exists $\eta(\epsilon)>0$, such that 
	\[\frac{1}{2} \left( \int_{B(\theta_0,\epsilon)^C} \left | - \sum_{i=1 , i \neq j}^{\infty} w^i q_i(\theta) \right | d\theta \right)^2 \geq \eta(\epsilon). \] 
	
	Therefore, it follows that 
	\begin{align*}
	\liminf_{n \to \infty}\ & D_\alpha(\pi(\theta|\nX)\| q'_n(\theta)) \geq \eta(\epsilon) >0  \quad P_{\theta_0}-a.s.
	\end{align*}
\end{proof}

In the following result, we show that if $q_i(\theta), i \in \{1,2,\ldots \}$ in the definition of $\{q'_n(\theta)\}$ in~\eqref{eq:eq14i} are Dirac delta distributions then \[ \liminf_{n \to \infty} D_\alpha(\pi(\theta|\nX)\| q'_n(\theta)) \geq 2 (1-w^j)^2 >0  \quad P_{\theta_0}-a.s, \]
where $w^j$ is the weight of $\delta_{\theta_0}$. Consider a sequence $\{q_n(\theta)\}$, that  converges weakly to a convex combination of $\delta_{\theta_i}, i\in \{1,2,\ldots \}$ such that for weights $\{w^i \in (0,1) : \sum_{i=1}^{\infty} w^i=1\},$
\begin{align}
q_n(\theta) \Rightarrow \sum_{i=1}^{\infty} w^i \delta_{\theta_i} ,
\label{eq:eq14}
\end{align}
where for any $j\in \{1,2,\ldots \}$  , $\theta_j = \theta_0$ and for all $i\in \{1,2,\ldots\} \backslash \{j\}$, $\theta_j \neq \theta_0$. 
\vspace{1em}
\begin{lemma}
	\label{lem:subopt}
	The $\alpha-$R\'enyi divergence between the true posterior and sequence $\{q_n(\theta)\}$ is bounded below by a positive number $2(1-w^j)^2$; that is,
	\[ \liminf_{n \to \infty} D_\alpha(\pi(\theta|\nX)\| q_n(\theta)) \geq   2(1-w^j)^2 >0 \quad P_{\theta_0}-a.s, \] 
where $w^j$	is the  weight  of $\delta_{\theta_0}$ in the definition of sequence $\{q_n(\theta)\}$. 
\end{lemma}  
\begin{proof}
\citet[Theorem 19]{TErvan2012} shows that for any $\alpha>0$, the $\alpha-$R\'enyi divergence  $D_\alpha(p(\theta)\| q(\theta))$ is a lower semi-continuous function of the pair $(p(\theta),q(\theta))$ in the weak topology on  the space of probability measures. Recall from~\eqref{eq:4} that the true posterior distribution $\pi(\theta|\nX)$ converges weakly to $\delta_{\theta_0}$, $P_{\theta_0}-a.s$. Using this fact it follows that
	\begin{align*} \liminf_{n \to \infty} D_\alpha(\pi(\theta|\nX)\| q_n(\theta)) \geq D_\alpha \left (\delta_{\theta_0} \bigg\| \sum_{i=1}^{\infty} w_i \delta_{\theta_i} \right )\quad P_{\theta_0}-a.s.
	\end{align*}
Next, using Pinsker's inequality~\citep{CoverTM2006} for  $\alpha > 1$, we have
\begin{align} 
\nonumber
 D_\alpha \left(\delta_{\theta_0} \bigg\| \sum_{i=1}^{\infty} w^i \delta_{\theta_i} \right )  &\geq \frac{1}{2} \left( \int_{\Theta} \left  | \delta_{\theta_0} - \sum_{i=1}^{\infty} w^i \delta_{\theta_i} \right | d\theta \right)^2
\\
\nonumber
&=  \frac{1}{2} \left( \int_{\Theta} \left|(1-w^j) \delta_{\theta_0} - \sum_{i=1 , i \neq j}^{\infty} w^i \delta_{\theta_i} \right| d\theta \right)^2
\\
\nonumber
&= \frac{1}{2} \left( \int_{B(\theta_0,\epsilon)} (1-w^j) |\delta_{\theta_0}| d\theta + \sum_{i=1 , i \neq j}^{\infty} w^i \int_{B(\theta_i,\epsilon)}|- \delta_{\theta_i} | d\theta \right)^2
\\
&= \frac{1}{2} \left(  (1-w^j) + \sum_{i=1 , i \neq j}^{\infty} w^i  \right)^2 = 2 (1-w^j)^2 ,
\end{align}

where $B(\theta_i,\epsilon)$ is the  ball of radius $\epsilon$ centered at $\theta_i$. Note that, there always exist an $\epsilon>0$, such that $\bigcap_{i=1}^{\infty} B(\theta_i,\epsilon) = \phi$. Since, by the definition  of sequence $\{q_n(\theta)\}$, $w^j\in(0,1)$, therefore $2(1-w^j)^2>0$ and the lemma follows.
	\end{proof}

 Now we show that any sequence of distributions $\{s_n(\theta)\} \subset \cQ$ that converges weakly to a distribution $s(\theta) \in \cQ$, that has positive density at any point other than the true parameter $\theta_0$, cannot achieve zero \scKL\ divergence in the limit.

 \begin{proof}[\textbf{Proof of Proposition~\ref{prop:UBfin}}]
     Observe that for any good sequence $\{\bar q_n(\theta)\}$ 
     \[\min_{q \in \mathcal Q}  D_{\alpha}(\pi(\theta|\nX) \|q(\theta)) \leq    D_{\alpha}(\pi(\theta|\nX)\|\bar{q}_n(\theta)). \] 
     Therefore, for the second part, it suffices to show that \[  D_{\alpha}(\pi(\theta|\nX)\|\bar{q}_n(\theta)) < B+ o_{P_{\theta_0}}(1).\] 
     The subsequent arguments in the proof are for any $n \geq \max(n_1,n_2,n_3,n_M)$, where $n_1,n_2$, and $n_3$ are defined in Assumption~\ref{def:gsequence}. First
     observe that, for any compact ball $K$ containing the true parameter $\theta_0$,
     \begin{align}
     \nonumber
     \frac{\alpha-1}{\a}&D_\a(\pi(\theta|\nX)\| \bar q_n(\theta)) 
     \\ 
     &= \frac{1}{\alpha} \log \left  ( \int_{K} \bar q_n(\theta) \left ( \frac{\pi(\theta|\nX)}{\bar q_n(\theta)} \right)^{\alpha} d \theta + \int_{\Theta \backslash K} \bar q_n(\theta) \left ( \frac{\pi(\theta|\nX)}{ \bar q_n(\theta)} \right)^{\alpha} d \theta \right).
     \label{eq:eq11a}
     \end{align}
     
     
     First, we approximate the first integral on the right hand side using the LAN condition in Assumption~\ref{assume:lan}. Let
     $\Delta_{n,\theta_0} := \sqrt{n}(\hat{\theta}_n - \theta_0)$, where $\hat{\theta}_n  \to \theta_0$, $P_{\theta_0} -a.s.$ and $\Delta_{n,\theta_0}$ converges in distribution to $\mathcal{N}(0,I(\theta_0)^{-1})$. Reparameterizing the expression with $\theta =
     \theta_0 + n^{-1/2} h$, we have
     \begin{align}
     \nonumber
     \int_{K} & \bar q_n(\theta)  \left ( \frac{\pi(\theta|\nX)}{ \bar q_n(\theta)} \right)^{\alpha} d \theta = n^{-1/2} \int_{K}  \bar q_n(\theta_0 + n^{-1/2}
     h)   \left(   \frac {\pi(\theta_0 + n^{-1/2}
         h) \prod_{i=1}^n \frac{p(X_i|(\theta_0  + n^{-1/2} h) )}{p(X_i|\theta_0)}}{ \bar q_n(\theta_0 + n^{-1/2}
         h) \int_{\Theta}   \prod_{i=1}^n \frac{p(X_i|\gamma)}{p(X_i|\theta_0)}   \pi(\gamma) d\gamma} \right)^{\a}  dh \\
     \nonumber
     &= n^{-1/2} \int_{K}  \bar q_n(\theta_0 + n^{-1/2}
     h)   \left(   \frac {\pi(\theta_0 + n^{-1/2}
         h) \prod_{i=1}^n \frac{p(X_i|(\theta_0  + n^{-1/2} h) )}{p(X_i|\theta_0)}}{ \bar q_n(\theta_0 + n^{-1/2}
         h) \int_{\Theta}   \prod_{i=1}^n \frac{p(X_i|\gamma)}{p(X_i|\theta_0)}   \pi(\gamma) d\gamma} \right)^{\a}  dh \\
     &= n^{-1/2} \int_{K}  \bar q_n(\theta_0 + n^{-1/2} h)
     \bigg ( \pi(\theta_0 + n^{-1/2} h) \frac{ \exp (h
         I(\theta_0)\Delta_{n,\theta_0} -	\frac{1}{2}h^2I(\theta_0)+
         o_{P_{\theta_0}}(1) )}{ \bar q_n(\theta_0 + n^{-1/2}
         h) \int_{\Theta}   \prod_{i=1}^n \frac{p(X_i|\gamma)}{p(X_i|\theta_0)}   \pi(\gamma) d\gamma} \bigg)^{\a} dh.
     \label{eq:eq11b}
     \end{align}
     {Resubstituting $h = \sqrt{ n}(\theta - \theta_0)$ in the expression above and reverting to the previous parametrization,}
     \begin{align}
     \nonumber
     &= \int_{K}  \bar q_n(\theta) \left ( \pi(\theta) \frac{ \exp \left(
         \sqrt{n}(\theta - \theta_0) I(\theta_0)\Delta_{n,\theta_0} -
         \frac{1}{2} n (\theta - \theta_0)^2 I(\theta_0) + o_{P_{\theta_0}}(1) \right) }{ \bar q_n(\theta)\int_{\Theta}   \prod_{i=1}^n \frac{p(X_i|\gamma)}{p(X_i|\theta_0)}   \pi(\gamma) d\gamma} \right)^{\alpha} d\theta 
     \\
     \nonumber
     &= \int_{K}  \bar q_n(\theta)  \left ( \pi(\theta) \frac{ e^{o_{P_{\theta_0}}(1)}  \exp \left( -\frac{1}{2} n I(\theta_0) \left((\theta - \theta_0)^2 - 2(\theta - \theta_0)(\hat{\theta}_n- \theta_0)\right) \right) }{ \bar q_n(\theta) \int_{\Theta}   \prod_{i=1}^n \frac{p(X_i|\gamma)}{p(X_i|\theta_0)}   \pi(\gamma) d\gamma } \right)^{\a} d\theta.
     \\
     \intertext{Completing the square by dividing and multiplying the numerator by $\exp \left( \frac{1}{2} n I(\theta_0) \left((\hat{\theta}_n - \theta_0 )^2 \right) \right)$}
     \nonumber
     &= \int_{K}  \bar q_n(\theta)  \left ( \pi(\theta) \frac{ e^{o_{P_{\theta_0}}(1)} \exp \left( \frac{1}{2} n I(\theta_0) \left((\hat{\theta}_n - \theta_0 )^2 \right) \right) \exp \left( -\frac{1}{2} n I(\theta_0) \left((\theta - \hat{\theta}_n)^2 \right) \right) }{ \bar q_n(\theta) \int_{\Theta}   \prod_{i=1}^n \frac{p(X_i|\gamma)}{p(X_i|\theta_0)}   \pi(\gamma) d\gamma } \right)^{\a} d\theta
     \\
     &= \int_{K}  \bar q_n(\theta)  \left ( \pi(\theta) \frac{ e^{o_{P_{\theta_0}}(1)} \exp \left( \frac{1}{2} n I(\theta_0) \left((\hat{\theta}_n - \theta_0 )^2 \right) \right) \sqrt{\frac{2\pi}{nI(\theta_0)}}  \mathcal{N}(\theta;\hat{\theta}_n,(n I(\theta_0))^{-1}) }{ \bar q_n(\theta) \int_{\Theta}   \prod_{i=1}^n \frac{p(X_i|\gamma)}{p(X_i|\theta_0)}   \pi(\gamma) d\gamma } \right)^{\a} d\theta,
     \label{eq:eq13}
     \end{align}
     where, in the last equality we used the definition of Gaussian density, $\mathcal{N}(\cdot;\hat{\theta}_n,(n I(\theta_0))^{-1})$. 
     
     Next, we approximate the integral in the denominator of~\eqref{eq:eq13}. Using Lemma~\ref{lem:norm} (in the appendix) it follows that, there exist a sequence of compact balls $\{K_n \subset \Theta\}$, such that $\theta_0 \in K_n$ and
     \begin{align}
     \nonumber
     \int_{\Theta} &  \prod_{i=1}^n \frac{p(X_i|\gamma)}{p(X_i|\theta_0)}   \pi(\gamma) d\gamma 
     \\
     =&   \sqrt{\frac{2\pi}{nI(\theta_0)}} e^{\left( \frac{1}{2} n I(\theta_0) \left( (\hat{\theta}_n - \theta_0 )^2 \right) \right)}    \bigg( e^{o_{P_{\theta_0}}(1)} \int_{K_n} \pi(\gamma) \mathcal{N}(\gamma;\hat{\theta}_n,(n I(\theta_0))^{-1}) d\gamma + o(1) \bigg).
     \label{eq:eqp6}
     \end{align}
     
     Substituting~\eqref{eq:eqp6} into~\eqref{eq:eq13}, we obtain 
     \vspace{-1em}
     \begin{align}
     \hspace{-.2in}\int_{K} & \bar q_n(\theta)  \left ( \frac{\pi(\theta|\nX)}{ \bar q_n(\theta)} \right)^{\alpha} d \theta
     = \int_{K}  \bar q_n(\theta)^{1-\a}  \left (  \frac{ e^{o_{P_{\theta_0}}(1)}  \pi(\theta) \mathcal{N}(\theta;\hat{\theta}_n,\frac{1}{n I(\theta_0)}) }{  \bigg( e^{o_{P_{\theta_0}}(1)} \int_{K_n} \pi(\gamma) \mathcal{N}(\gamma;\hat{\theta}_n,\frac{1}{n I(\theta_0)}) d\gamma + o(1) \bigg) } \right)^{\a} d\theta .
     \label{eq:eq14a}
     \end{align}
     Now, recall the definition of compact ball $K$, $n_1$ and $n_2$ from Assumption~\ref{def:gsequence} and fix $n \geq n'_0$, where $n'_0=\max(n_1,n_2)$. Note that $n_2$ is chosen, such that for all $n \geq n_2$, the bound in Assumption~\ref{def:gsequence}(3) holds on the set $\Theta \backslash K$. Next, consider the second term inside the logarithm function on the right hand side of~\eqref{eq:eq11a}. Using Assumption~\ref{def:gsequence}(3), we obtain
     \vspace{-.3em}
     \begin{align}
     \int_{\Theta \backslash K}   \bar q_n(\theta) \left ( \frac{\pi(\theta|\nX)}{ \bar q_n(\theta)} \right)^{\alpha} d\theta \leq M_r^{\alpha} \int_{\Theta \backslash K}   \bar q_n(\theta)  d\theta \quad~P_{\theta_0}-a.s.
     \label{eq:eq13a}
     \end{align}
     
     Recall that the \good\ $\{ \bar q_n(\cdot)\}$ exists $P_{\theta_0}-a.s$ with mean $\hat{\theta}_n$, for all $n\geq n_1$ and therefore it converges weakly to $\delta_{\theta_0}$ (Assumption~\ref{def:gsequence}(2)). Combined with the fact that compact set $K$ contains the true parameter $\theta_0$, it follows that the second term in~\eqref{eq:eq11a} is of $o(1)$, $P_{\theta_0}-a.s$. 
     Therefore, the second term inside the logarithm function on the right hand side of~\eqref{eq:eq11a} is $o(1)$:
     \begin{align}
     \int_{\Theta \backslash K}  \bar q_n(\theta) \left ( \frac{\pi(\theta|\nX)}{ \bar q_n(\theta)} \right)^{\alpha} d\theta  = o(1)~P_{\theta_0}-\text{a.s}.
     \label{eq:eq14b}
     \end{align}
     
     Substituting~\eqref{eq:eq14a} and~\eqref{eq:eq14b} into~\eqref{eq:eq11a}, we have
     \begin{align}
     \nonumber
     \frac{\alpha-1}{\a}&D_\a(\pi(\theta|\nX)\| \bar q_n(\theta)) 
     \\ 
     \nonumber
     &= \frac{1}{\alpha} \log \left  ( \int_{K} \bar q_n(\theta)^{1-\a}  \left (  \frac{ e^{o_{P_{\theta_0}}(1)} \pi(\theta)  \mathcal{N}(\theta;\hat{\theta}_n,(n I(\theta_0))^{-1})  }{ \bigg( e^{o_{P_{\theta_0}}(1)} \int_{K_n} \pi(\gamma) \mathcal{N}(\gamma;\hat{\theta}_n,(n I(\theta_0))^{-1}) d\gamma + o(1) \bigg)  }  \right)^{\a} d\theta + o(1) \right)
     \\
     \nonumber
     &= \frac{1}{\alpha} \log \left  (  e^{o_{P_{\theta_0}}(1)} \int_{K} \bar q_n(\theta)^{1-\a}  \left (  \frac{ \pi(\theta)  \mathcal{N}(\theta;\hat{\theta}_n,(n I(\theta_0))^{-1})  }{    \bigg( e^{o_{P_{\theta_0}}(1)} \int_{K_n} \pi(\gamma) \mathcal{N}(\gamma;\hat{\theta}_n,(n I(\theta_0))^{-1}) d\gamma + o(1) \bigg)   }  \right)^{\a} d\theta + o(1) \right). \tag{$\star\star$}
     \intertext{Now observe that,}
     \nonumber
     (\star\star)&\sim  \frac{1}{\alpha} \log \left  (  \int_{K} \bar q_n(\theta)^{1-\a}  \left (  \frac{ \pi(\theta)  \mathcal{N}(\theta;\hat{\theta}_n,(n I(\theta_0))^{-1})  }{    \bigg( e^{o_{P_{\theta_0}}(1)} \int_{K_n} \pi(\gamma) \mathcal{N}(\gamma;\hat{\theta}_n,(n I(\theta_0))^{-1}) d\gamma + o(1) \bigg)   }  \right)^{\a} d\theta  \right) 
     \\
     \nonumber
     &= \frac{1}{\alpha} \log \left  (  \int_{K} \bar q_n(\theta)^{1-\a} \pi(\theta)^{\a}  \mathcal{N}(\theta;\hat{\theta}_n,(n I(\theta_0))^{-1})^{\a}  d\theta  \right)  
     \\
     \nonumber
     & \quad - \log \bigg( e^{o_{P_{\theta_0}}(1)} \int_{K_n} \pi(\gamma) \mathcal{N}(\gamma;\hat{\theta}_n,(n I(\theta_0))^{-1}) d\gamma + o(1) \bigg)
     \nonumber
     \\
     \nonumber
     &\sim \frac{1}{\alpha} \log \left  (  \int_{K} \bar q_n(\theta)^{1-\a} \pi(\theta)^{\a}  \mathcal{N}(\theta;\hat{\theta}_n,(n I(\theta_0))^{-1})^{\a}  d\theta  \right)  
     \\
     & \quad - \log \bigg(  \int_{K_n} \pi(\gamma) \mathcal{N}(\gamma;\hat{\theta}_n,(n I(\theta_0))^{-1}) d\gamma  \bigg) \blue{+ o_{P_{\theta_0}}(1)}.
     \label{eq:eq15a}
     \end{align}
     
     Note that  
     $ \left( \mathcal{N}(\theta;\hat{\theta}_n,(n I(\theta_0))^{-1}) \right)^{\a} = \left( \sqrt{ \frac{n I(\theta_0)}{2 \pi } } \right)^{\alpha} \left( \sqrt{\frac{2\pi}{n \alpha I(\theta_0)}}  \right)  \mathcal{N}(\theta;\hat{\theta}_n,(n \alpha I(\theta_0))^{-1}). $ 
     
     Substituting this into \eqref{eq:eq15a}, for large enough $n$, we have
     \begin{align}
     \nonumber
     \frac{\alpha-1}{\a}&D_\a(\pi(\theta|\nX)\| \bar q_n(\theta)) 
     \\
     \nonumber
     \sim &  \frac{\a-1}{2 \a} \log n - \frac{\log \a}{2\a}  +\frac{\a-1}{2 \a} \log \frac{I(\theta_0)}{2 \pi } + \frac{1}{\alpha} \log \int_{K} \bar q_n(\theta)^{1-\a} \pi(\theta)^{\a} \mathcal{N}(\theta;\hat{\theta}_n,(n \alpha I(\theta_0))^{-1})  d\theta
     \\
     &   
     - \log  \left(  \int_{K_n}  \pi(\gamma)  \mathcal{N}(\gamma;\hat{\theta}_n,(n I(\theta_0))^{-1})  d\gamma \right) . 
     \label{eq:eq16a}
     \end{align}
     
     From the Laplace approximation (Lemma~\ref{lem:laplace}) and the continuity of the logarithm, we have 
     \begin{align}
     \nonumber
     \frac{1}{\alpha} \log \int_{K} \bar q_n(\theta)^{1-\a} \pi(\theta)^{\a} \mathcal{N}(\theta;\hat{\theta}_n,(n \alpha I(\theta_0))^{-1})  d\theta
     \sim  \frac{1-\alpha}{\alpha} \log {\bar q_n(\hat{\theta}_n)} +  \log \pi(\hat{\theta}_n) .
     \end{align}
     Next, using the Laplace approximation~(Lemma~\ref{lem:laplace}) on the last term in~\eqref{eq:eq16a} yields
     \[- \log  \left(   \int_{K_n}  \pi(\gamma)  \mathcal{N}(\gamma;\hat{\theta}_n,(n I(\theta_0))^{-1})  d\gamma \right) \sim - \log  \left(  \pi(\hat{\theta}_n)   \right).\]
     
     Substituting the above two approximations into~\eqref{eq:eq16a}, for large enough $n$, we obtain
     \begin{align}
     \nonumber
     \frac{\alpha-1}{\a}&D_\a(\pi(\theta|\nX)\| \bar q_n(\theta)) 
     \\
     \nonumber
     \sim &  
     \frac{1 - \alpha}{\alpha} \log \bar q_n(\hat{\theta}_n) + \log    \pi(\hat{\theta}_n)    - \frac{\log \a}{2\a} +\frac{\a-1}{2 \a} \log \frac{I(\theta_0)}{2 \pi } +
     \frac{\alpha -1}{2\alpha} \log {n}  - \log \pi(\hat{\theta}_n) \blue{ + o_{P_{\theta_0}}(1)}
     \\
     = &   \frac{1 - \alpha}{\alpha} \log \bar q_n(\hat{\theta}_n)  - \frac{\log \a}{2\a} +\frac{\a-1}{2 \a} \log \frac{I(\theta_0)}{2 \pi } +
     \frac{\alpha -1}{2\alpha} \log {n} \blue{ + o_{P_{\theta_0}}(1)}.
     \label{eq:eq17a}
     \end{align}
     
     Now, recall Assumption~\ref{def:gsequence}(4) which, combined with the monotonicity of logarithm function, implies that $\log \bar q_n(\cdot)$ is concave for all $n\geq n_3$. Using Jensen's inequality, \[ \log \bar q_n(\hat{\theta}_n) =  \log \bar q_n \left( \int \theta \bar q_n(\theta) d\theta \right) \geq \int  \bar q_n(\theta) \log \bar q_n(\theta) d\theta.  \]
     Since $\a>1$, 
     \[\frac{1 - \alpha}{\alpha} \log \bar q_n(\hat{\theta}_n) \leq - \frac{\a-1}{\a} \int  \bar q_n(\theta) \log \bar q_n(\theta) d\theta.  \]
     Using Lemma~\ref{lem:entropy},  there exists $n_M\geq 1$ and $0 < \bar M <\infty$, such that for all $n\geq  n_M$
     \begin{align} - \frac{\a-1}{\a} \int  \bar q_n(\theta) \log \bar q_n(\theta) d\theta \leq \frac{\a-1}{2\a} \log \left( 2\pi \bar e \frac{\bar M}{n} \right) = \frac{\a-1}{2\a} \log (2\pi \bar e {\bar M} ) - \frac{\a-1}{2\a} \log n, 
     \label{eq:eq18}
     \end{align}
     where $\bar e$ is the Euler's constant.
     Substituting~\eqref{eq:eq18} into the right hand  side of~\eqref{eq:eq17a}, we have for all $n \geq n_0$, where $n_0=\max(n_0',n_3,n_M)$, 
     \begin{align}
     \nonumber
     \frac{1 - \alpha}{\alpha}& \log \bar q_n(\hat{\theta}_n)  - \frac{\log \a}{2\a} +\frac{\a-1}{2 \a} \log \frac{I(\theta_0)}{2 \pi } +
     \frac{\alpha -1}{2\alpha} \log {n}.
     \\
     \nonumber
     \leq &   \frac{\a-1}{2\a} \log (2\pi \bar e {\bar M} ) - \frac{\a-1}{2\a} \log n - \frac{\log \a}{2\a} +\frac{\a-1}{2 \a} \log \frac{I(\theta_0)}{2 \pi } +
     \frac{\alpha -1}{2\alpha} \log {n}
     \nonumber
     \\
     \nonumber
     =& \frac{\a-1}{2\a} \log (2\pi \bar e {\bar M} )  - \frac{\log \a}{2\a} +\frac{\a-1}{2 \a} \log \frac{I(\theta_0)}{2 \pi }
     \\
     =& \frac{\alpha-1}{\a} \frac{1}{2} \log \frac{\bar e \bar M I(\theta_0) }{\alpha^{\frac{1}{\a-1}}}.
     \label{eq:eq19a}
     \end{align}
     
     Observe that the left hand side in~\eqref{eq:eq17a} is always non-negative, implying the right hand side must be too for large $n$. Therefore, 
     the following inequality must hold for all $n\geq n_0$:
     \[ \frac{\bar e \bar M I(\theta_0)} {\alpha^{\frac{1}{\a-1}}}  \geq 1.  \] 
     Consequently, substituting~\eqref{eq:eq19a} into~\eqref{eq:eq17a}, we have 
     \begin{align}
     D_\a(\pi(\theta|\nX)\| \bar q_n(\theta)) &\leq\frac{1}{2} \log \frac{\bar e \bar M I(\theta_0) }{\alpha^{\frac{1}{\a-1}}} \blue{+ o_{P_{\theta_0}}(1)}\  \forall n\geq n_0,
     \label{eq:eq20a}
     \end{align}
     and the result follows.

 \end{proof}

\vspace{0.5em}
We next state an important inequality,  that is a direct consequence of H\"older's inequality. We use the following result in the proof of Lemma~\ref{thm:degen}.
\vspace{0.5em}
\begin{lemma}\label{lem:hold}
	For any set $K \subset \Theta$ and $\alpha>1$ and any sequence of distributions $\{q_n(\theta)\} \subset \cQ$, the following inequality holds true
	\begin{align}
	\int_{\Theta} q_n(\theta) \left( \frac{\pi(\theta|\nX)}{q_n(\theta)} \right)^{\alpha} d\theta \geq \frac{ \left( \int_{K} \pi(\theta|\nX) d\theta \right) ^{\alpha} }{ \left( \int_{K} q_n(\theta) d\theta \right) ^{\alpha-1} }.
	\label{eq:eqa12b}
	\end{align}
\end{lemma}
\begin{proof}
	Fix a set $K \subset \Theta$. Since $\alpha>1$, using H\"older's inequality for $f(\theta) = \frac{\pi(\theta|\nX)}{q_n(\theta)^{1- \frac{1}{\alpha}}}$ and $g(\theta) = q_n(\theta)^{1- \frac{1}{\alpha}}$, 
	\begin{align*}
	\int_{K} \pi(\theta|\nX) d\theta &= \int_{K} f(\theta) g(\theta)  d\theta \\
	&\leq   \left( \int_{K} \frac{\pi(\theta|\nX)^{\alpha}}{q_n(\theta)^{\alpha-1}} d\theta \right) ^{\frac{1}{\alpha}} \left( \int_{K} q_n(\theta) d\theta \right) ^{1 - \frac{1}{\alpha}}.
	\end{align*}
	
	It is straightforward to observe from the above equation that,
	\begin{align*}
	\int_{K} \frac{\pi(\theta|\nX)^{\alpha}}{q_n(\theta)^{\alpha-1}} d\theta \geq \frac{ \left( \int_{K} \pi(\theta|\nX) d\theta \right) ^{\alpha} }{ \left( \int_{K} q_n(\theta) d\theta \right) ^{\alpha-1} }.
	\end{align*}
	
	Also note that, for any set $K$, the following inequality holds true, 
	\begin{align}
	\int_{\Theta} q_n(\theta) \left( \frac{\pi(\theta|\nX)}{q_n(\theta)} \right)^{\alpha} d\theta \geq \int_{K} \frac{\pi(\theta|\nX)^{\alpha}}{q_n(\theta)^{\alpha-1}} d\theta \geq \frac{ \left( \int_{K} \pi(\theta|\nX) d\theta \right) ^{\alpha} }{ \left( \int_{K} q_n(\theta) d\theta \right) ^{\alpha-1} },
	\label{eq:eqa12}
	\end{align}
	and the result follows immediately.
\end{proof}

\begin{proof}[\textbf{Proof of Lemma~\ref{thm:degen}}]
    First, we fix $n \geq 1$ and let $M_r$ be a sequence such that $M_r \to \infty $ as $r \to \infty.$ Recall that $\hat\theta_n$ is the maximum likelihood estimate and denote $\tilde{\theta}_n = \bbE_{q_n(\theta)}[\theta]$. Define a set 
    \[K_r := \{\theta \in \Theta : |\theta-\hat{\theta}_n| > M_r \} \bigcup \{\theta \in \Theta : |\theta-\tilde{\theta}_n| > M_r \}. \]
    Now, using Lemma~\ref{lem:hold} with $K = K_r$, we have
    \begin{align}
    \int_{\Theta} q_n(\theta) \left( \frac{\pi(\theta|\nX)}{q_n(\theta)} \right)^{\alpha} d\theta \geq \frac{ \left( \int_{K_r} \pi(\theta|\nX) d\theta \right) ^{\alpha} }{ \left( \int_{K_r} q_n(\theta) d\theta \right) ^{\alpha-1} }. 
    \label{eq:eq1}
    \end{align}
    
    Note that the left hand side in the above equation does not depend on $r$ and when $r \to \infty$ both the numerator and denominator on the right hand side converges to zero individually. For the ratio to diverge, however, we require the denominator to converge much faster than the numerator. To be more precise, observe that for a given $n$, 
    since $\alpha-1 < \alpha$ the tails of $q_n(\theta)$ must decay significantly faster than the tails of the true posterior for the right hand side in~\eqref{eq:eq1} to diverge as $r \to \infty$. 
    
    We next show that there exists an $n_0\geq 1$ such that for all $n \geq n_0$, the right hand side in~\eqref{eq:eq1} diverges as $r \to \infty$. 
    Since the posterior distribution satisfies the Bernstein-von Mises Theorem \citep{vdV00}, we have
    %
    \[  \int_{K_r} \pi(\theta|\nX) d\theta  =   \int_{K_r} \mathcal{N}(\theta;\hat \theta_n, (n I(\theta_0))^{-1}) d\theta + o_{P_{\theta_0}}(1). \]  
    Observe that the numerator on the right hand side of~\eqref{eq:eq1} satisfies,
    \begin{align}
    \nonumber
    \bigg( \int_{K_r}  & \pi(\theta|\nX) d\theta \bigg) ^{\alpha} 
    \nonumber
    =  \left( \int_{K_r} \mathcal{N}(\theta;\hat\theta_n, (n I(\theta_0))^{-1}) d\theta + o_{P_{\theta_0}}(1) \right) ^{\alpha}
    \\
    \nonumber
    &\geq \left( \int_{ \{ |\theta-\hat{\theta}_n| > M_r\}} \mathcal{N}(\theta;\hat\theta_n, (n I(\theta_0))^{-1}) d\theta + o_{P_{\theta_0}}(1) \right) ^{\alpha}
    \\
    \nonumber
    &= \left( \int_{ \{ \theta-\hat{\theta}_n > M_r \} } \mathcal{N}(\theta;\hat\theta_n, (n I(\theta_0))^{-1}) d\theta + \int_{ \{\theta-\hat{\theta}_n \leq - M_r \} } \mathcal{N}(\theta;\hat\theta_n, (n I(\theta_0))^{-1}) d\theta + o_{P_{\theta_0}}(1) \right) ^{\alpha}
    \\
    & \geq \left( \int_{ \{ \theta-\hat{\theta}_n > M_r \}} \mathcal{N}(\theta;\hat\theta_n, (n I(\theta_0))^{-1}) d\theta  + o_{P_{\theta_0}}(1) \right) ^{\alpha}.
    \end{align}
    Now, using the lower bound on the Gaussian tail distributions from~\cite{feller1968}
    \begin{align}
    \nonumber
    \left( \int_{K_r} \pi(\theta|\nX) d\theta \right) ^{\alpha} &=  \left( \int_{K_r} \mathcal{N}(\theta;\hat\theta_n, (n I(\theta_0))^{-1}) d\theta + o_{P_{\theta_0}}(1) \right) ^{\alpha}
    \\
    \nonumber
    & \geq \left( \frac{1}{\sqrt{2\pi}}\left( \frac{1}{\sqrt{n I(\theta_0)} M_r } - \frac{1}{(\sqrt{n I(\theta_0)} M_r)^3} \right) e^{-\frac{n I(\theta_0)}{2} M_r^2 } + o_{P_{\theta_0}}(1) \right) ^{\alpha}
    \\
    &\sim \left( \frac{1}{\sqrt{2\pi}} \frac{1}{\sqrt{n I(\theta_0)} M_r }  e^{-\frac{n I(\theta_0)}{2} M_r^2 } + o_{P_{\theta_0}}(1) \right) ^{\alpha},
    \label{eq:eq2}
    \end{align}
    where the last approximation follows from the fact that, for large $r$, \[ \left( \frac{1}{\sqrt{n I(\theta_0)} M_r } - \frac{1}{(\sqrt{n I(\theta_0)} M_r)^3} \right)  \sim \frac{1}{\sqrt{n I(\theta_0)} M_r }. \]
    
    Next, consider the denominator on the right hand side of~\eqref{eq:eq1}. Using the union bound
    \begin{align}
    \left( \int_{K_r} q_n(\theta) d\theta \right) ^{\alpha-1} \leq \left( \int_{ \{|\theta-\tilde{\theta}_n| > M_r \}} q_n(\theta) d\theta  + \int_{ \{|\theta-\hat{\theta}_n| > M_r \} } q_n(\theta) d\theta  \right) ^{\alpha-1}.
    \label{eq:eq5}
    \end{align}
    
    Since, $\tilde{\theta}_n$ and $\hat{\theta}_n$ are finite for all $n\geq 1$, there exists an $\e>0$ such that for large $n$,
    \( |\tilde\theta_n - \hat\theta_n| \leq \e. \) Applying the triangle inequality, 
    \[ |\theta - \hat\theta_n| \leq | \theta - \tilde\theta_n| + | \tilde\theta_n - \hat\theta_n| \leq  | \theta - \tilde\theta_n| + \e. \]
    
    Therefore, $\{|\theta-\hat{\theta}_n| > M_r \} \subseteq \{|\theta-\tilde{\theta}_n| > M_r - \e \} $ and it follows from~\eqref{eq:eq5} that
    \begin{align}
    \nonumber
    \left( \int_{K_r} q_n(\theta) d\theta \right) ^{\alpha-1} 
    & \leq \left( \int_{ \{|\theta-\tilde{\theta}_n| > M_r \}} q_n(\theta) d\theta  + \int_{ \{|\theta-\tilde{\theta}_n| > M_r - \e \} } q_n(\theta) d\theta  \right) ^{\alpha-1}.
    \end{align}
    
    Next, using the sub-Gaussian tail distribution bound from \cite[Theorem 2.1]{BoLuMa2013}, 
    \begin{align}
    \left( \int_{ \{|\theta-\tilde{\theta}_n| > M_r \}} q_n(\theta) d\theta  + \int_{ \{|\theta-\tilde{\theta}_n| > M_r - \e \} } q_n(\theta) d\theta  \right) ^{\alpha-1} \leq \left( 2 e^{-\frac{\gamma^2_n M_r^2}{2B}} + 2 e^{-\frac{\gamma^2_n (M_r-\e)^2}{2B}}  \right) ^{\alpha-1}.
    \label{eq:eq6}
    \end{align}
    
    For large $r$, \(M_r \sim M_r-\e \), and it follows that
    \begin{align}
    \left( \int_{ \{|\theta-\tilde{\theta}_n| > M_r \}} q_n(\theta) d\theta  + \int_{ \{|\theta-\tilde{\theta}_n| > M_r - \e \} } q_n(\theta) d\theta  \right) ^{\alpha-1} \lesssim \left( 4 e^{-\frac{\gamma^2_n M_r^2}{2B}}  \right) ^{\alpha-1}.
    \label{eq:eq3}
    \end{align}
    Substituting~\eqref{eq:eq2} and~\eqref{eq:eq3} into~\eqref{eq:eq1}, we obtain
    \begin{align*}
    \int_{\Theta} q_n(\theta) \left( \frac{\pi(\theta|\nX)}{q_n(\theta)} \right)^{\alpha} d\theta \gtrsim \left( \frac{ \frac{1}{\sqrt{2\pi}} \frac{1}{\sqrt{n I(\theta_0)} M_r }  e^{-\frac{n I(\theta_0)}{2} M_r^2 }   + o_{P_{\theta_0}}(1)  } {\left( 4 e^{-\frac{\gamma^2_n M_r^2}{2B}}  \right) ^\frac{\alpha-1}{\alpha}}  \right) ^{\alpha},
    \label{eq:eq4}
    \end{align*}
    for large $r$. Observe that
    \begin{align}
    \frac{ \frac{1}{\sqrt{2\pi}} \frac{1}{\sqrt{n I(\theta_0)} M_r }  e^{-\frac{n I(\theta_0)}{2} M_r^2 }    } {\left( 4 e^{-\frac{\gamma^2_n M_r^2}{2B}}  \right) ^\frac{\alpha-1}{\alpha}}  &=  \frac{1}{4^\frac{\alpha-1}{\alpha} \sqrt{2\pi}} \frac{1}{ M_r }  \left( \frac{1}{\sqrt{n I(\theta_0)} }e^{M_r^2\left( \frac{\alpha-1}{\alpha}\frac{\gamma^2_n }{2B}  - \frac{n I(\theta_0)}{2} \right)  } \right).
    \end{align}
    
    Since $\gamma_n^2 > n$, choosing $n_0 = \min \left\{n : \left( \frac{\alpha-1}{\alpha}\frac{\gamma^2_n }{2B}  - \frac{n I(\theta_0)}{2} \right) > 0 \right\}$ implies that for all $n \geq n_0$, as $r \to \infty$, the left hand side in~\eqref{eq:eq4} diverges and the result follows.
    
\end{proof}

\subsection{Proofs in Section~\ref{sec:asymptoteEP}}

\begin{proof}[\textbf{Proof of Lemma~\ref{lem:subopt1ep}}]
    \citet[Theorem 1]{PosnerE.1975Rcsf} shows that, the \scKL\ divergence  $\scKL(p(\theta)\| s(\theta))$ is a lower semi-continuous function of the pair $(p(\theta),s(\theta))$ in the weak topology on  the space of probability measures. Recall from~\eqref{eq:4} that the true posterior distribution $\pi(\theta|\nX)$ converges weakly to $\delta_{\theta_0}$, $P_{\theta_0}-a.s$. Using this fact it follows that
    \begin{align*} \liminf_{n \to \infty} \scKL(\pi(\theta|\nX) \| s_n(\theta)) \geq \scKL \left(\delta_{\theta_0} \| s(\theta) \right)\quad P_{\theta_0}-a.s.
    \end{align*}
    Next, using Pinsker's inequality~\cite{CoverTM2006} for  $\alpha > 1$, we have
    \begin{align*} 
    \scKL \left(\delta_{\theta_0} \| s(\theta) \right) &\geq \frac{1}{2} \left( \int_{\Theta} \left | \delta_{\theta_0} - s(\theta) \right| d\theta \right)^2.
    \end{align*}
    Now, fixing $\epsilon > 0$ such that $s(\theta)$ has positive density in the complement of the ball of radius $\epsilon$ centered at $\theta_0$, $B(\theta_0,\epsilon)^C$, we have  
    \begin{align}
    \nonumber
    \liminf_{n \to \infty} \scKL(\pi(\theta|\nX) \| s_n(\theta)) 
    &\geq \frac{1}{2} \left( \int_{B(\theta_0,\epsilon)}  \left | \delta_{\theta_0} - s(\theta) \right| d\theta + \int_{B(\theta_0,\epsilon)^C} \left | \delta_{\theta_0} - s(\theta) \right | d\theta \right)^2
    \\
    \nonumber
    &\geq \frac{1}{2} \left( \int_{B(\theta_0,\epsilon)^C} \left | \delta_{\theta_0} - s(\theta) \right | d\theta \right)^2
    \\
    &=  \frac{1}{2} \left( \int_{B(\theta_0,\epsilon)^C} \left |  - s(\theta) \right | d\theta \right)^2 \quad P_{\theta_0}-a.s.
    \end{align}
    Since $s(\theta)$ has positive density in the set $B(\theta_0,\epsilon)^C$, there exists $\eta(\epsilon)>0$, such that 
    \[\frac{1}{2} \left( \int_{B(\theta_0,\epsilon)^C} \left |  - s(\theta) \right | d\theta \right)^2 \geq \eta(\epsilon), \] 
    completing the proof.
\end{proof}

\subsection{Proofs in Section~\ref{sec:latent}}
\begin{proof}[\textbf{Proof of Lemma~\ref{lem:LAN}}]
    We prove the assertion of the Lemma for the class of local latent parameters $z_i$ that have discrete and finite support.
    First observe that for $\alpha>1$, using Jensen's inequality
    \begin{align}
    M(\nX|\theta)^{\alpha} = \min_{q(\nZ)\in \cQ^n} \int_{ \sZ^n} q(\nZ) \left(\frac{p( \nZ, \mathbf X_n|\theta)}{q(\nZ) } \right)^{\alpha}  d\nZ \geq \left[\int_{ \sZ^n} p( \nZ, \mathbf X_n|\theta)   d\nZ \right]^{\alpha}.
    \label{eq:lb}
    \end{align}
    Now since family $\cQ^n$ contains point masses, we choose a member of family $\cQ^n$ which is a joint distribution of point masses at $\nZ^p:=\{z_1^p,z_2^p,\ldots,z_n^p\}$ to obtain
    \begin{align}
    M(\nX|\theta)^{\alpha} = \min_{q(\nZ)\in \cQ^n} \int_{ \sZ^n} q(\nZ) \left(\frac{p( \nZ, \mathbf X_n|\theta)}{q(\nZ) } \right)^{\alpha}  d\nZ \leq   \left[p( \nZ^p, \mathbf X_n|\theta)  \right]^{\alpha},
    \label{eq:ub}
    \end{align}
    where $\nZ^p$ is as defined in Assumption~\ref{assume:clp}.
    
    Since, $f(x)=x^\alpha$ is increasing for $\alpha>1$ and $x>0$, it follows from~\eqref{eq:lb},~\eqref{eq:ub}, and monotonicity of the logarithm function that
    \begin{align}
    \log \int_{ \sZ^n} p( \nZ, \mathbf X_n|\theta)   d\nZ  \leq \log M(\nX|\theta)  \leq \log p( \nZ^p, \mathbf X_n|\theta). 
    \label{eq:ul}
    \end{align}
    
    Now using Assumption~\ref{assume:clp}~(1) and~(2(ii)), that is $d_{H}(z_0,\nZ^p)= o(\rho_n)$, it follows that 
    at some rate $\rho_n$ with $\rho_n \downarrow 0 $ and $n\rho_n^2 \to \infty$; that is for all bounded, stochastic $h_n= O_{P_{0}}(1) $, 
    \begin{align*} 
    \int_{\{\nZ: d_{H}(\nZ,z_0) \geq \rho_n  \}}& p(z_{1:n}|\nX,\theta = \theta_0+ n^{-1/2}h_n ) d\nZ \\
    \leq &\int_{\{\nZ: d_{H}(\nZ,\nZ^p) + d_{H}(z_0,\nZ^p) \geq \rho_n  \}} p(z_{1:n}|\nX,\theta = \theta_0+ n^{-1/2}h_n ) d\nZ
    \\
     \leq &\int_{\{\nZ: d_{H}(\nZ,\nZ^p)  \geq \rho_n(1-\e)  \}} p(z_{1:n}|\nX,\theta = \theta_0+ n^{-1/2}h_n ) d\nZ = o_{P_{0}}(1),
    \end{align*} 
    where the first inequality follows from using the fact that $d_{H}(\nZ,z_0) \leq d_{H}(\nZ,\nZ^p) + d_{H}(z_0,\nZ^p)$, the second inequality uses the fact that $d_{H}(z_0,\nZ^p)= o(\rho_n)$, that is for some $\e\in (0,1)$, $d_{H}(z_0,\nZ^p)<\e \rho_n $ for sufficiently large $n$, and the last inequality is due to Assumption~\ref{assume:clp}~(1).
   
    Therefore, it can be observed from the above result that  the conditioned latent posterior $p(z_{1:n}|\nX,\theta_0)$ concentrates at $z_0$. Consequently, when the local latent parameters are discrete it follows that
    \begin{align*}
    \log \int_{ \sZ^n} p( \nZ, \mathbf X_n|\theta_0)   d\nZ  = \log \int_{ \sZ^n} \frac{p(z_{1:n}|\nX,\theta_0)}{p(z_{1:n}|\nX,\theta_0)} p( \nZ, \mathbf X_n|\theta_0)   d\nZ =\log  p( z_0, \mathbf X_n|\theta_0) + o_{P_0}(1).
    \end{align*}
    Now it  follows that
    \begin{align} \log M(\nX|\theta_0)  =  \log p( z_0, \mathbf X_n|\theta_0) + o_{P_0}(1)= \log \int_{ \sZ^n} p( \nZ, \mathbf X_n|\theta_0)   d\nZ  + o_{P_0}(1).
    \label{eq:ap} 
    \end{align}
    
    Subtracting  $\log M(\nX|\theta_0)$ from~\eqref{eq:ul} and using~\eqref{eq:ap} yields
    \begin{align}
    \log \frac{\int_{ \sZ^n} p( \nZ, \mathbf X_n|\theta)   d\nZ }{\int_{ \sZ^n} p( \nZ, \mathbf X_n|\theta_0)   d\nZ } +o_{P_0}(1) \leq \log \frac{M(\nX|\theta)}{M(\nX|\theta_0)} \leq \log \frac{p( z_0, \mathbf X_n|\theta)}{p( z_0, \mathbf X_n|\theta_0)} + o_{P_0}(1).
    \end{align}
    Now, substituting $\theta= \theta_0 + n^{-1/2}h_n$ for all bounded and stochastic $h_n= O_{P_{0}}(1) $, and using the result in~\citet[Theorem 4.2]{bickel2012semiparametric} under the conditions in Assumption~\ref{assume:clp} the RHS and LHS above have the same LAN expansion and the result follows. Notice that, by definition, the s-LAN condition in~Assumption~\ref{assume:lan} is also true at $\nZ=\nZ^p$. Assumption~\ref{assume:clp}~(2(ii)) implies $d_{H}(z_0,\nZ^p)= o(\rho_n) $ with $\rho_n \downarrow 0 $ and $n\rho_n^2 \to \infty$, so that 
    \[  \log \frac{P^n_{\theta_0 + n^{-1/2} h_n,\nZ^p}}{P^n_{\theta_0 ,\nZ^p}} = \log \frac{P^n_{\theta_0 + n^{-1/2} h_n,z_0}}{P^n_{\theta_0 ,z_0}} + o(1).\]
Therefore, $\log \frac{p( z_0, \mathbf X_n|\theta_0 + n^{-1/2} h_n)}{p( z_0, \mathbf X_n|\theta_0)} = \log \frac{p(  \mathbf X_n|z_0,\theta_0 + n^{-1/2} h_n)}{p(  \mathbf X_n|z_0,\theta_0)} + \log \frac{p(z_0|\theta_0 + n^{-1/2} h_n)}{p(z_0|\theta_0)} = \log \frac{P^n_{\theta_0 + n^{-1/2} h_n,z_0}}{P^n_{\theta_0 ,z_0}}+ o(1)$ also have the same expansion as given in the s-LAN condition in~Assumption~\ref{assume:lan}. 
\end{proof}

\begin{proof}[\textbf{Proof of Proposition~\ref{prop:UBfinL}}]
    Observe that for any good sequence $\{\bar q_n(\theta)\}$ and  $q(z_{1:n})$ as point masses (discrete distribution) at the truth $\nZ^0:=\{z_1^0,z_2^0,\ldots,z_n^0\}$, we have
    \begin{align}
    \nonumber
    \min_{q \in \mathcal Q} \min_{q(z_{1:n})\in \cQ^n }  &D_{\alpha}(\pi(\theta,z_{1:n}|\nX) \|q(\theta)q(z_{1:n})) 
    \\ 
    \nonumber
    &=    \min_{q(\theta)\in \bar \cQ, q(\nZ)\in \cQ^n} \frac{1}{\alpha-1}
    \log \int_{\Theta \times \sZ^n} q(\theta)q(\nZ)   \left(\frac{p(\theta, \nZ, \mathbf X_n)}{p(\nX)q(\theta)q(\nZ) } \right)^{\alpha} d\theta d\nZ
    \\ 
    \nonumber
    &
    \leq     \frac{1}{\alpha-1}
    \log \int_{\Theta } \bar q_n(\theta)  \left(\frac{p(\theta, \nZ^0, \mathbf X_n)}{p(\nX)\bar q_n(\theta) } \right)^{\alpha} d\theta 
    \\
    &
    \leq     \frac{1}{\alpha-1}
    \log \int_{\Theta } \bar q_n(\theta)  \left(\frac{\pi(\theta, \nZ^0| \mathbf X_n)}{\bar q_n(\theta) } \right)^{\alpha} d\theta . 
    \label{eq:eqb1}
    \end{align} 
    
    Also note that, using the definition of $\pi(\theta, \nZ^0| \mathbf X_n)$, we have
    \begin{align}
    \pi(\theta, \nZ^0| \mathbf X_n) = \frac{\pi(\theta){\pi(\nZ^0|\theta)}p(\mathbf X_n|\theta,\nZ^0)}{\int_{\Theta \times \sZ^n} \pi(\theta){\pi(\nZ|\theta)}p(\mathbf X_n|\theta,\nZ) d\theta d\nZ} \leq \frac{\pi(\theta){\pi(\nZ^0|\theta)}p(\mathbf X_n|\theta,\nZ^0)}{\int_{\Theta} \pi(\theta){\pi(\nZ^0|\theta)}p(\mathbf X_n|\theta,\nZ^0) d\theta },
    \label{eq:eqb2}
    \end{align}
    where the second inequality follows from the fact that $\nZ$ is a discrete random variable. Therefore substituting~\eqref{eq:eqb2} into~\eqref{eq:eqb1} yields
    \begin{align}
    \nonumber
    \min_{q \in \mathcal Q} \min_{q(z_{1:n})\in \cQ^n }  D_{\alpha}(\pi(\theta,z_{1:n}|\nX) \|q(\theta)q(z_{1:n})) &\leq \frac{1}{\alpha-1}
    \log \int_{\Theta } \bar q_n(\theta)  \left(\frac{\pi(\theta)p(\mathbf X_n,\nZ^0|\theta)}{\bar q_n(\theta) \int_{\Theta} \pi(\theta)p(\mathbf X_n,\nZ^0|\theta) d\theta } \right)^{\alpha} d\theta 
    \\
    \nonumber
    &= \frac{1}{\alpha-1}
    \log \int_{\Theta } \bar q_n(\theta)  \left(\frac{\pi(\theta|\nX,\nZ^0)}{\bar q_n(\theta) } \right)^{\alpha} d\theta 
    \\
    &=: D_{\alpha}(\pi(\theta|\nX,\nZ^0)\|\bar{q}_n(\theta)).  
    \end{align} 
    Therefore, for the second part, it suffices to show that \[  D_{\alpha}(\pi(\theta|\nX,\nZ^0)\|\bar{q}_n(\theta)) < B + o_{P_0}(1).\] 
    
    The subsequent arguments in the proof are for any $n \geq \max(n_1,n_2,n_3,n_M)$, where $n_1,n_2$, and $n_3$ are defined in Assumption~\ref{def:gsequence}. First
    observe that, for any compact ball $K$ containing the true parameter $\theta_0$,
    \begin{align}
    \nonumber
    \frac{\alpha-1}{\a}&D_\a(\pi(\theta|\nX,\nZ^0)\| \bar q_n(\theta)) 
    \\ 
    &= \frac{1}{\alpha} \log \left  ( \int_{K} \bar q_n(\theta) \left ( \frac{\pi(\theta|\nX,\nZ^0)}{\bar q_n(\theta)} \right)^{\alpha} d \theta + \int_{\Theta \backslash K} \bar q_n(\theta) \left ( \frac{\pi(\theta|\nX,\nZ^0)}{ \bar q_n(\theta)} \right)^{\alpha} d \theta \right).
    \label{eq:eqb11a}
    \end{align}
    
    
    First, we approximate the first integral on the right hand side using the LAN condition in Assumption~\ref{assume:lan}. \textcolor{black}{Let
    $\Delta_{n,(\theta_0,z_0)} := \sqrt{n}(\hat{\theta}_n - \theta_0)$, where $\hat{\theta}_n  \to \theta_0$, $P_0 -a.s.$ and $\Delta_{n,(\theta_0,z_0)}$ converges in distribution to $\mathcal{N}(0,I(\theta_0,z_0)^{-1})$}~\citep[Lemma 25.23 and 25.25]{vdV00}. Now the proof follows similar steps as used in the proof of Proposition~\ref{prop:UBfin}. 
\end{proof}



\bibliographystyle{plainnat}
\bibliography{refs}
\end{document}